\documentclass{article}
\usepackage{amsmath,amssymb,enumerate,theorem,epsfig}

\newcommand{\qed}{\hfill$\Box$\medskip}

\newcommand{\pr}{\operatorname{pr}}

\def\G{\Gamma}

\def\cJ{{\mathcal J}}

\def\cC{{\mathcal C}}
\def\cD{{\mathcal D}}

\def\C{{\mathbb C}}

\def\Z{{\mathbb Z}}

\def\pr{\prime}
\def\pt{{\rm pt}}

\newtheorem{theorem}{Theorem}[section]

\newtheorem{lemma}[theorem]{Lemma}
\newtheorem{proposition}[theorem]{Proposition}

{ \theorembodyfont{\rmfamily}
\newtheorem{remark}[theorem]{Remark}

}

\newenvironment{proof}{{\noindent\it Proof.}}{}

\begin{document}
\title{Symplectic fillings of links of quotient surface singularities}
\author{Mohan Bhupal and Kaoru Ono\thanks
{Partly supported by the Grant-in-Aid in Scientific Research No.\
14340019, Japan Society for the Promotion of Science.}}

\date{}
\maketitle

\begin{abstract}
We study symplectic deformation types of minimal symplectic fillings
of links of quotient surface singularities.  In particular, there
are only finitely many symplectic deformation types for each
quotient surface singularity.
\end{abstract}


\section{Introduction}
\label{intro} In recent years, the geometry of contact structures on
three manifolds has been a subject of intensive studies. In
particular, tight contact structures have been the focus of
interest. For instance, tight contact structures on lens spaces have
now been classified by Giroux and Honda. The link $L$ of an isolated
surface singularity $(V,O)$ provides examples of tight contact three
manifolds. Namely, the complex tangency to the link gives a
codimension one distribution $\xi = TL \cap J(TL)$ which is
completely non-integrable, hence a contact structure. Here $J$ is
the complex structure on $V \setminus O$. Let $\pi\colon
\widetilde{V} \to V$ be a resolution of the singularity and $U$ a
neighborhood of $O$ in  $V$ such that $\partial U =L$. Then
$\pi^{-1}(U)$ is a so-called symplectic filling of $(L, \xi)$ and
$\xi$ is a symplectically fillable contact structure, which implies
that $\xi$ is tight by a theorem of Eliashberg and Gromov.

It is also interesting to classify symplectic fillings of the links
of certain classes of isolated surface singularities. For the case
of cyclic quotient singularities of $A_{n,1}$-type, McDuff
classified symplectic deformation classes of minimal symplectic
fillings \cite{M}. H.~Ohta and the second named author investigated
the cases of simple singularities \cite{OO2} and simple elliptic
singularities \cite{OO4}. Meanwhile, Lisca \cite{L} presents a
classification for the case of cyclic quotient singularities. In
this paper, we study the case of quotient surface singularities $\C
/ \Gamma$, where $\Gamma$ is a finite subgroup of $Gl(2;\C)$. Note
that this class contains all simple singularities, which is  the
case where $\Gamma \subset Sl(2;\C)$. Simple singularities are
characterized as isolated surface singularities which are described
by both quotient singularities and hypersurface singularities. Thus
we can use both aspects in the argument. Namely, since they are
quotient singularities, the link is a spherical space form. In
particular, they carry a metric of positive scalar curvature. This
fact is one of main ingredients in \cite{OO1}. They are also
hypersurface singularities with explicit defining equations. This
enables us to describe the compactifications of their Milnor fibres
in appropriate weighted projective spaces (K.~Saito). The results in
\cite{OO2} are some of the main ingredients in this paper. Since the
situation here is more complicated than in the case of simple
singularities, we have to study rational curves with negative
self-intersection numbers carefully. The main theorem is the
following.

\begin{theorem}
A symplectic filling of the link of a quotient surface singularity
is symplectic deformation equivalent to the complement of a certain
divisor in an iterated blow-up of ${\C}P^2$ or ${\C}P^1 \times
{\C}P^1$.
\end{theorem}

A detailed description of the symplectic fillings is given later. In
particular, we get finiteness of symplectic deformation types of
minimal symplectic fillings for each quotient surface singularity.


\section{Preliminaries}
\label{sec:2} Our basic strategy is the following. Firstly, find an
appropriate strong concave filling $Y$ of the link of the
singularity and glue it with a given symplectic filling $X$ to get a
closed symplectic $4$-manifold $Z$. Secondly, use a rationality
criterion for symplectic $4$-manifolds in order to show that $Z$ is
rational. Thirdly, study how $Y$ is embedded in the rational
symplectic $4$-manifold $Z$. Eventually, $Y$ is chosen as a regular
neighborhood of a certain divisor $K$ in a blow-up of a rational
ruled surface. (We call $K$ a compactifying divisor.) So the third
step is replaced by the study of embeddings of $K$ in $Z$.

In this section, we present several facts which are necessary to
carry out these steps. First of all, we recall some basic results
for rational and ruled symplectic $4$-manifolds.

\begin{theorem}[McDuff \cite{M}] \label{thm:McDuff}
Let $(M, \omega)$ be a closed symplectic $4$-manifold. If $M$
contains a symplectically embedded $2$-sphere $C$ of nonnegative
self-intersection number $k$, then $M$ is either a rational
symplectic $4$-manifold or a blow-up of a ruled symplectic
$4$-manifold. In particular, if $k=0$ (resp.\ $1$), $M$ becomes a
ruled symplectic $4$-manifold (resp.\ the complex projective plane)
after blowing down symplectic $(-1)$-curves away from $C$.
\end{theorem}

Here a rational symplectic $4$-manifold means a symplectic blow-up
of the complex projective plane at some points, and a ruled
symplectic $4$-manifold means a $2$-sphere bundles over an oriented
surface with a symplectic structure which is nondegenerate on each
fibre.  Combining Theorem~2.1 and Taubes' theorem ``SW=Gr'', we get
the following.

\begin{theorem}[Ohta--Ono \cite{OO}, Liu \cite{Li}]
Let $(M, \omega )$ be a closed symplectic $4$-manifold such that
$\int_M c_1(M) \wedge \omega >0.$ Then $(M, \omega )$ is either a
rational symplectic $4$-manifold or a (blow-up of a) ruled
symplectic $4$-manifold.
\end{theorem}

If the pseudo\-holomorphic curve $C$ is singular, we have the
following result as a byproduct of uniqueness of minimal symplectic
fillings of the link of a simple singularity \cite{OO2}.

\begin{theorem}[Ohta--Ono \cite{OO3}] \label{thm:Ohta-Ono}
Let $M$ be a closed symplectic $4$-manifold containing a
pseudoholomorphic rational curve $C$ with a $(2,3)$-cusp point.
Suppose that $C$ is nonsingular away from the $(2,3)$-cusp point. If
the self-intersection number $C^2$ of $C$ is positive, then $M$ must
be a rational symplectic $4$-manifold and $C^2$ is at most $9$.
Moreover, if $M \setminus C$ does not contain any symplectic
$(-1)$-curves, then $C$ represents the Poincar\'e dual to $c_1(M)$,
that is, an anti-canonical divisor. When $C^2=9$, $M={\C}P^2$ and
$C$ is a pseudoholomorphic cuspidal cubic curve.
\end{theorem}

We call a homology class $e \in H_2(M;\Z)$ a symplectic $(-1)$-class
if $e$ is represented by a symplectically embedded $2$-sphere of
self-intersection number $-1$. A symplectic $(-1)$-curve class is
represented by a $J$-holomorphic sphere for a generic compatible (or
tame) almost complex structure $J$. However, if we restrict the
class of compatible (or tame) almost complex structure, this may not
be the case. Here we have the following (essentially Proposition~4.1
in \cite{OO3}).

\begin{proposition} \label{prop:2.4} Let $M$ be a symplectic $4$-manifold and
$C_1, \dots ,C_k$ irreducible $J_0$-holomorphic curves in $M$ with
respect to a compatible (or tame) almost complex structure $J_0$.
Suppose that each of $C_1, \dots ,C_k$ is either nondegenerate,
singular or of higher genus. Then, for a generic $J$ among
compatible (or tame) almost complex structures for which $C_1,
\dots, C_k$ are pseudoholomorphic, any symplectic $(-1)$-curve has a
unique $J$-holomorphic representative.
\end{proposition}

Here we call $C_i$ nondegenerate if the linearized operator of the
pseudoholomorphic curve equation is surjective at $C_i$. Now we
collect a series of observations.

If a pseudoholomorphic curve $C$ intersects a $(-1)$-curve
transversally, Lemma~4.1 in \cite{OO2} ensures that the image under
the blowing down map is also pseudoholomorphic with respect to a
suitable almost complex structure. Transversality of intersections
can be achieved by a small perturbation of the almost complex
structure.

\begin{lemma} \label{lem:2.5} Let $M$ be a closed symplectic
$4$-manifold and $L$ a symplectically embedded $2$-sphere of
self-intersection number $1$. Then any irreducible singular or
higher genus pseudoholomorphic curve $C$ in $M$ satisfies $C \cdot L
\geq 3.$ In particular, neither an irreducible singular nor a higher
genus pseudoholomorphic curve is contained in $M \setminus L$.
\end{lemma}

\begin{proof}  If necessary, we perturb the almost complex structure
slightly in such a way that the $(-1)$-curves do not pass through
the singular points of $C$. We then blow down a maximal disjoint
family of pseudoholomorphic $(-1)$-curves away from $L$. Here we can
assume that $L$ and $C$ are also pseudoholomorphic with respect to
the same almost complex structure (Proposition~\ref{prop:2.4}). Then
$M$ becomes the complex projective plane and $L$ becomes a line.
Since $C$ is singular and irreducible or of higher genus, the image
$\overline{C}$ has degree at least 3. Thus we have $C \cdot L =
\overline{C} \cdot L \geq 3.$ \qed
\end{proof}

\begin{lemma} \label{lem:2.6} Let $M$ be a closed symplectic
$4$-manifold and $C$ a pseudoholomorphic  rational curve with a
$(2,3)$-cusp point as a unique singularity. Suppose that the
self-intersection number of $C$ is positive. Then neither an
irreducible singular nor a higher genus pseudoholomorphic curve is
contained in $M \setminus C$.
\end{lemma}

\begin{proof}  Suppose that $D$ is such a singular or a higher genus
pseudoholomorphic curve. Let $J$ be a compatible almost complex
structure on $M$ with respect to which $C$ and $D$ are
pseudoholomorphic. By Proposition~\ref{prop:2.4}, we may assume that
all symplectic $(-1)$-classes are represented by $J$-holomorphic
$(-1)$-curves. We blow down a maximal family of $J$-holomorphic
$(-1)$-curves in $M \setminus C$. Thus we may assume that $M
\setminus C$ does not contain any symplectic $(-1)$-curves. By
Theorem~\ref{thm:Ohta-Ono}, $C$ is an anti-canonical divisor and we
have $c_1(M)[D]=0$. If $D$ is singular or of higher genus, the
adjunction formula tells us that $D \cdot D \geq 0$. On the other
hand, the intersection form on $M \setminus C$ is negative definite.
Hence $[D]$ is homologous to zero, which is absurd. \qed
\end{proof}

\begin{lemma} \label{lem:2.7}
Let $M$ and $C$ be as in Lemma~\ref{lem:2.6}. Suppose that the
self-intersection number of $C$ is at least 2. Then there does not
exist a pseudo\-holomorphic curve $A$ such that $A$ is either
singular and irreducible or of higher genus and such that $A \cdot C
= 1$.
\end{lemma}

\begin{proof}  If such a curve $A$ exists, we blow up $M$ at the
intersection point of $A$ and $C$. Then the proper transform of $A$
violates the conclusion of Lemma~\ref{lem:2.6}. \qed
\end{proof}

If the self-intersection number of $C$ is $1$, there exist singular
or genus $1$ pseudoholomorphic curves $A$ such that $A \cdot C =1$.
In addition, if $M \setminus C$ is minimal, it turns out that $A$ is
homologous to $C$ in $M$.

\begin{lemma} \label{lem:2.8}
Let $M$ be a closed symplectic $4$ manifold and $L$ a symplectically
embedded sphere of self-intersection number $1$. Then no
symplectically embedded sphere of nonnegative self-intersection
number is contained in $M \setminus L$. Pseudoholomorphic
$(-1)$-curves in $M \setminus L$ are mutually disjoint.
\end{lemma}

\begin{lemma} \label{lem:2.9}
Let $M$ be a closed symplectic $4$-manifold and $C$ an irreducible
singular or higher genus pseudo\-holomorphic curve. Then no
symplectically embedded sphere of nonnegative self-intersection
number is contained in $M\setminus C$.
\end{lemma}

\begin{proof}  Let $A$ be a symplectically embedded sphere in $M \setminus
C$. Set $k= A \cdot A$. If $k=1$, the result follows from
Lemma~\ref{lem:2.5}. If $k > 1$, blow up $M$ at $k-1$ points on $A$.
The proper transform of $A$ has self-intersection number $1$. So
this case is reduced to the case where $k=1$. If $k=0$, we blow down
a maximal disjoint family of pseudoholomorphic $(-1)$-curves away
from $A$.  Note that Proposition~\ref{prop:2.4} guarantees that
these $(-1)$-curves and $C$ are pseudoholomorphic with respect to
the same almost complex structure. The blown-down manifold is a
ruled symplectic $4$-manifold and $A$ is a fibre. Since $C$ is
singular and irreducible or of higher genus, its image
$\overline{C}$ under the blowing down map is also singular and
irreducible or of higher genus. Thus it is not a fibre of the
ruling. Since fibres sweep out the whole space, $\overline{C}$
should intersect a fibre. This contradicts the fact that
$\overline{C} \cdot A = C \cdot A =0.$ \qed
\end{proof}

Similarly, we get the following.

\begin{lemma} \label{lem:2.10} Let $M$ be a closed symplectic
$4$-manifold and $C$ a singular pseudoholomorphic curve. Then there
is no symplectically embedded sphere $A$ of nonnegative
self-intersection number such that $A \cdot C =1$.
\end{lemma}

\begin{proof} If $k= A \cdot A$ is positive, we blow up $M$ at $k$
points on $A \setminus C$. So we may assume that $k=0$. We blow down
a maximal disjoint family of pseudoholomorphic $(-1)$-curves away
from $A$ to get a ruled symplectic $4$-manifold. Then the image
$\overline{C}$ of $C$ under the blowing down map satisfies
$\overline{C} \cdot A = 1$. However, there exists another fibre $A'$
passing through a singular point of $C$, for which we have
$\overline{C} \cdot A' \geq 2$. This is a contradiction. \qed
\end{proof}

The following lemma is a consequence of Theorem~1 in \cite{OO1}.

\begin{lemma} \label{lem:2.11} Let $X$ be a symplectic filling
of the link of a simple singularity. Then pseudoholomorphic
$(-1)$-curves in $X$ are mutually  disjoint.
\end{lemma}

\begin{proof} Suppose that there are two pseudoholomorphic $(-1)$-curves $E$
and $E'$ which intersect each other. Contract $E$ to get a
symplectic manifold $\pi\colon X \to X'$. Then the homology class
$\pi_*[E']$ has nonnegative self-intersection number.  Note that
$X'$ is also a symplectic filling of $\partial X$. This contradicts
the fact that any symplectic filling of the link of a simple
singularity has negative definite intersection form (Theorem~1 in
\cite{OO1}). \qed
\end{proof}

\begin{remark} \label{rem:2.12} In Lemma~\ref{lem:2.5}, Lemma~\ref{lem:2.6}
and Lemma~\ref{lem:2.7}, we showed the non-existence of higher genus
pseudo\-holomorphic curves in the complement of a certain divisor
$D$.  These arguments also imply the non-existence of a cycle of
pseudo\-holomorphic spheres in the complement of $D$.  Indeed, a
cycle of pseudoholomorphic spheres is a stable map of genus $1$. By
gluing the adjacent components around the nodes we get an
irreducible symplectically embedded surface of genus 1 which is
pseudoholomorphic with respect to a compatible almost complex
structure which coincides the original almost complex structure
outside of a neighbourhood of the nodes. But Lemma~\ref{lem:2.5},
Lemma~\ref{lem:2.6} and Lemma~\ref{lem:2.7} prohibit the existence
of such a pseudoholomorphic curve of genus $1$.
\end{remark}

\begin{remark}
In \cite{OO2}, we used the fact that the canonical bundles of the
minimal symplectic fillings of the links of simple singularities are
trivial. In the case of types $E_6,E_7$ and $E_8$, we used K3
surfaces to find an appropriate compactification.  This argument is
also applied to cases of types $A_n$ and $D_n$ with $n$ small. For
general $A_n$ and $D_n$, this fact was shown by Kanda. Here we note
that in the case of $A_n$ and $D_n$, the compactification contains a
pseudoholomorphic rational curve $A$ of nonnegative
self-intersection number. Hence Theorem~\ref{thm:McDuff} due to
McDuff implies that the compactification is a rational or ruled
symplectic 4-manifold.  If it is ruled, the existence of a
pseudoholomorphic cuspidal rational curve implies that it is not
irrationally ruled. Hence the compactification is a rational
symplectic 4-manifold. Since the pseudoholomorphic cuspidal curve is
singular, Proposition~\ref{prop:2.4} guarantees an almost complex
structure with respect to which the cuspidal curve and a maximal
family of $(-1)$-curves away from $A$ are pseudoholomorphic. By
Lemma~4.1 in \cite{OO2}, the cuspidal curve becomes
pseudoholomorphic after blowing down. Since we know the final
picture, we can conclude that the canonical bundle of the complement
of the compactifying divisor is trivial.
\end{remark}


\section{Quotient singularities}
\label{sec:3} We consider germs of quotient singularities
$(\C^2/G,0)$, where $G$ is a finite subgroup of $Gl(2,\C)$. It is
known that every such quotient singularity is isomorphic to the
quotient of $\C^2$ by a small group $G<Gl(2,\C)$, where ``small''
means that $G$ does not contain any reflections. Also, it is known
that, for small groups $G_1,G_2$, the singularity $(\C^2/G_1,0)$ is
analytically isomorphic to $(\C^2/G_2,0)$ if and only if $G_1$ is
conjugate to $G_2$. Hence the problem of classifying quotient
singularities $(\C^2/G,0)$ is reduced to the problem of classifying
small subgroups of $Gl(2,\C)$ up to conjugation. Since $G$ is
finite, we may assume that $G \subset U(2)$. The action of $G$ on
$\C^2$ lifts to an action on the blow-up of $\C^2$ at the origin:
$p\colon \widehat{\C}^2 \to \C^2$. The exceptional divisor $E$ is
stable under the $G$ action which is induced by $G \subset U(2) \to
PU(2) \cong SO(3)$. The image of $G$ in $SO(3)$ is (conjugate to)
either a cyclic subgroup, a dihedral subgroup, the tetrahedral
subgroup, the octahedral subgroup or the icosahedral subgroup. The
quotient space $\widehat{\C}^2/G$ has isolated singularities at
$p(E^G)$.  Each of them is a cyclic quotient singularity. The end
result of this is in \cite{Br}. Briefly, quotient singularities can
be divided into five families, namely: cyclic quotient
singularities, dihedral singularities, tetrahedral singularities,
octahedral singularities and icosahedral singularities. Presently,
we describe the possible minimal resolutions that occur for quotient
singularities together with compactifying divisors which are
convenient from our point of view. The latter can be obtained by the
method of McCarthy and Wolfson \cite{MW}.

\subsubsection*{{\it 1. Cyclic quotient singularities, $A_{n,q}$, where
$0<q<n$ and $(n,q)=1$}} It is well-known that the minimal resolution
of $A_{n,q}$ is given by
\begin{center}
\epsfig{file=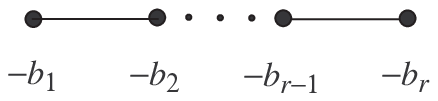}
\end{center}
where the $b_i$ are defined by the Hirzebruch--Jung continued
fraction:
$$ \frac nq = [b_1,b_2,\ldots,b_r] = b_1-
\cfrac{1}{b_2- \cfrac{1}{\ddots- \cfrac{1}{b_r}}}, \quad b_i\geq 2
\text{ for all $i$}. $$ It is not difficult to check that the
following configuration of curves gives a compactifying divisor for
$A_{n,q}$.
\begin{center}
\epsfig{file=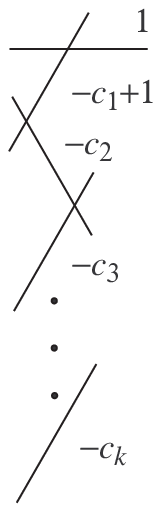}
\end{center}
where the $c_i$ are given by
$$ \frac {n}{n-q}=[c_1,c_2,\ldots,c_k], \quad c_i\geq 2,\text{ for all $i$}.$$

\subsubsection*{{\it 2. Dihedral singularities, $D_{n,q}$, where $1<q<n$
and $(n,q)=1$}} The minimal resolution is given by
\begin{center}
\epsfig{file=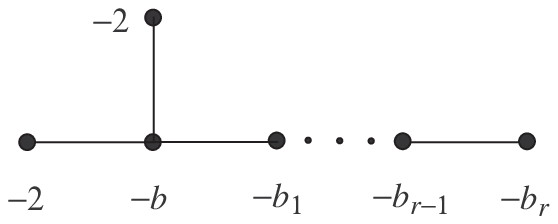}
\end{center}
where $b,b_i,i=1,\ldots,r$ are defined by
$$ \frac nq = [b,b_1,\ldots,b_r],\quad b\geq2,\ b_i\geq 2,\text{ for all $i$}. $$
In this case, one can check that a compactifying divisor is given by
\begin{center}
\epsfig{file=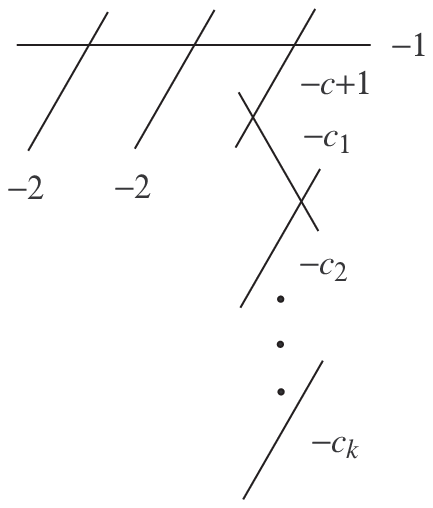}
\end{center}
where $c,c_i,i=1,\ldots,k$ are given by
$$ \frac {n}{n-q} = [c,c_1,\ldots,c_k],\quad c\geq2,\ c_i\geq 2,\text{ for all $i$}. $$

\subsubsection*{{\it 3. Tetrahedral singularities, $T_m$, where $m=1,3,5
\mod 6$}} The dual resolution graphs and compactifying divisors are
given in Table~\ref{tab:tet}.

\begin{table}
\caption{Dual resolution graphs and compactifying divisors for
tetrahedral singularities.} \label{tab:tet}
\begin{tabular}{|c|c|c|}
\hline
& Dual resolution graph & Compactifying divisor \\
\hline $m=6(b-2)+1$ &
{\begin{tabular}{c}\hspace{-.5cm}\epsfig{file=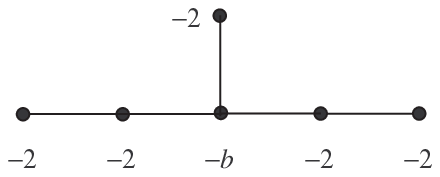}\end{tabular}}
\hspace{-.5cm} &
{\begin{tabular}{c}\hspace{-.5cm}\epsfig{file=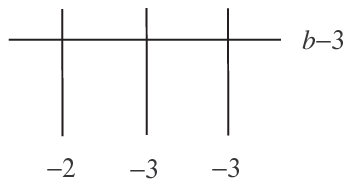}\end{tabular}}
\hspace{-.6cm} \\
\hline $m=6(b-2)+3$ &
{\begin{tabular}{c}\hspace{-.5cm}\epsfig{file=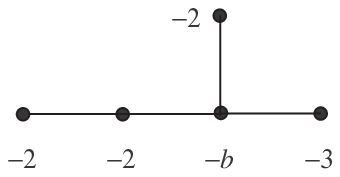}\end{tabular}}
\hspace{-.5cm} &
{\begin{tabular}{c}\hspace{-.5cm}\epsfig{file=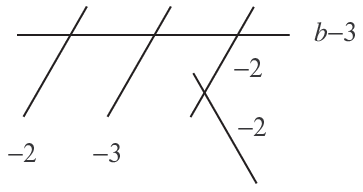}\end{tabular}}
\hspace{-.6cm} \\
\hline $m=6(b-2)+5$ &
{\begin{tabular}{c}\hspace{-.5cm}\epsfig{file=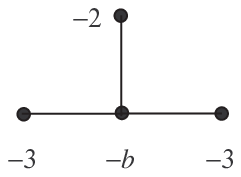}\end{tabular}}
\hspace{-.5cm} &
{\begin{tabular}{c}\hspace{-.5cm}\epsfig{file=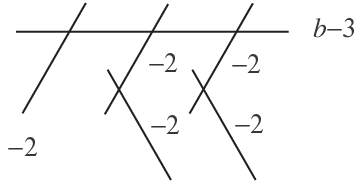}\end{tabular}}
\hspace{-.6cm} \\
\hline
\end{tabular}
\end{table}

\subsubsection*{{\it 4. Octahedral singularities, $O_m$, where
$(m,6)=1$}} The dual resolution graphs and compactifying divisors
are given in Table~\ref{tab:oct}.

\begin{table}
\caption{Dual resolution graphs and compactifying divisors for
octahedral singularities.} \label{tab:oct}
\begin{tabular}{|c|c|c|}
\hline
& Dual resolution graph & Compactifying divisor \\
\hline $m=12(b-2)+1$ &
{\begin{tabular}{c}\hspace{-.5cm}\epsfig{file=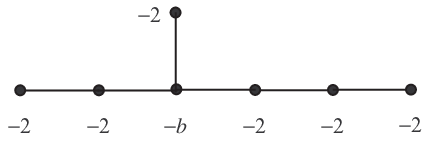}\end{tabular}}
\hspace{-.5cm} &
{\begin{tabular}{c}\hspace{-.5cm}\epsfig{file=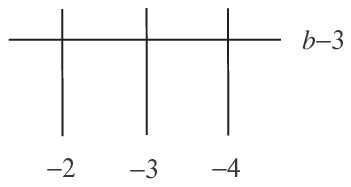}\end{tabular}}
\hspace{-.6cm} \\
\hline $m=12(b-2)+5$ &
{\begin{tabular}{c}\hspace{-.5cm}\epsfig{file=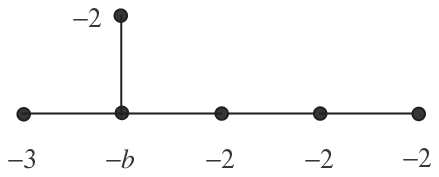}\end{tabular}}
\hspace{-.5cm} &
{\begin{tabular}{c}\hspace{-.5cm}\epsfig{file=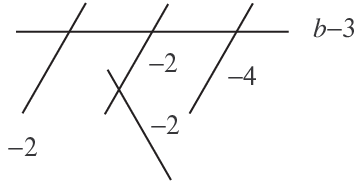}\end{tabular}}
\hspace{-.6cm} \\
\hline $m=12(b-2)+7$ &
{\begin{tabular}{c}\hspace{-.5cm}\epsfig{file=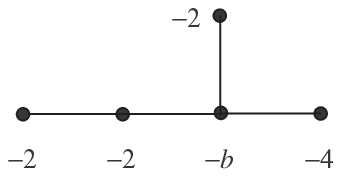}\end{tabular}}
\hspace{-.5cm} &
{\begin{tabular}{c}\hspace{-.5cm}\epsfig{file=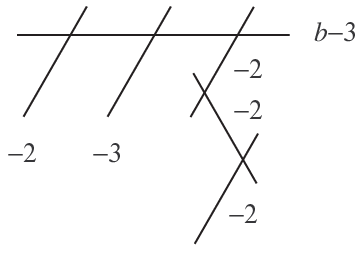}\end{tabular}}
\hspace{-.6cm} \\
\hline $m=12(b-2)+11$ &
{\begin{tabular}{c}\hspace{-.5cm}\epsfig{file=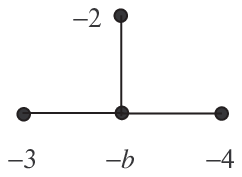}\end{tabular}}
\hspace{-.5cm} &
{\begin{tabular}{c}\hspace{-.5cm}\epsfig{file=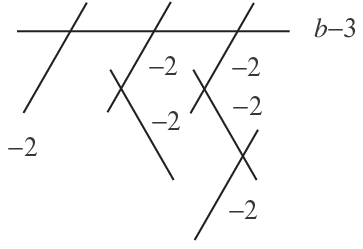}\end{tabular}}
\hspace{-.6cm} \\
\hline
\end{tabular}
\end{table}

\subsubsection*{{\it 5. Icosahedral singularities, $I_m$, where
$(m,30)=1$}} The dual resolution graphs and compactifying divisors
are given in Tables~\ref{tab:icos} and \ref{tab:icos2}.

\begin{table}
\caption{Dual resolution graphs and compactifying divisors for
icosahedral singularities.} \label{tab:icos}
\begin{tabular}{|c|c|c|}
\hline
& Dual resolution graph & Compactifying divisor \\
\hline $m=30(b-2)+1$ &
{\begin{tabular}{c}\hspace{-.5cm}\epsfig{file=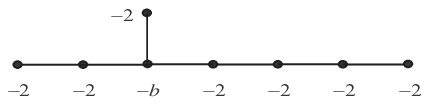}\end{tabular}}
\hspace{-.5cm} &
{\begin{tabular}{c}\hspace{-.5cm}\epsfig{file=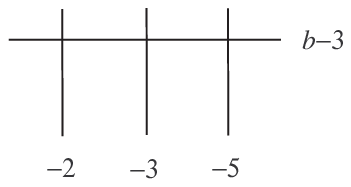}\end{tabular}}
\hspace{-.6cm} \\
\hline $m=30(b-2)+7$ &
{\begin{tabular}{c}\hspace{-.5cm}\epsfig{file=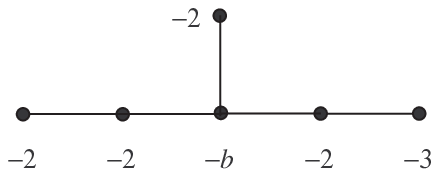}\end{tabular}}
\hspace{-.5cm} &
{\begin{tabular}{c}\hspace{-.5cm}\epsfig{file=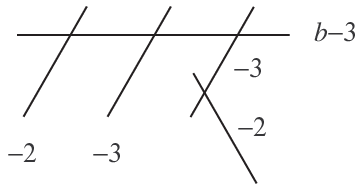}\end{tabular}}
\hspace{-.6cm} \\
\hline $m=30(b-2)+11$ &
{\begin{tabular}{c}\hspace{-.5cm}\epsfig{file=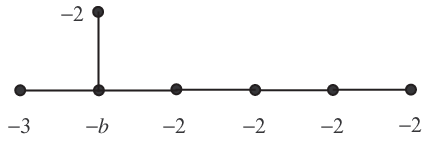}\end{tabular}}
\hspace{-.5cm} &
{\begin{tabular}{c}\hspace{-.5cm}\epsfig{file=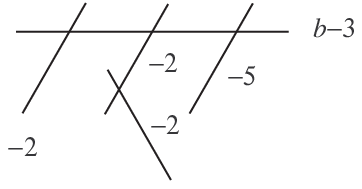}\end{tabular}}
\hspace{-.6cm} \\
\hline $m=30(b-2)+13$ &
{\begin{tabular}{c}\hspace{-.5cm}\epsfig{file=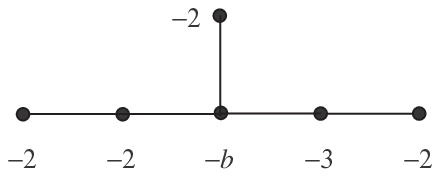}\end{tabular}}
\hspace{-.5cm} &
{\begin{tabular}{c}\hspace{-.5cm}\epsfig{file=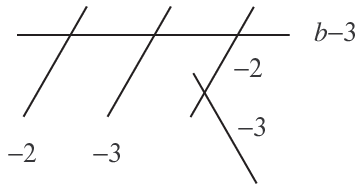}\end{tabular}}
\hspace{-.6cm} \\
\hline $m=30(b-2)+17$ &
{\begin{tabular}{c}\hspace{-.5cm}\epsfig{file=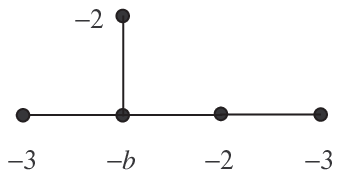}\end{tabular}}
\hspace{-.5cm} &
{\begin{tabular}{c}\hspace{-.5cm}\epsfig{file=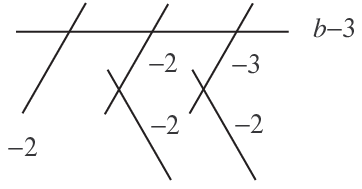}\end{tabular}}
\hspace{-.6cm} \\
\hline $m=30(b-2)+19$ &
{\begin{tabular}{c}\hspace{-.5cm}\epsfig{file=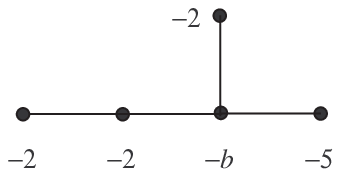}\end{tabular}}
\hspace{-.5cm} &
{\begin{tabular}{c}\hspace{-.5cm}\epsfig{file=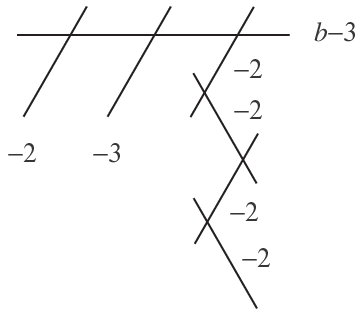}\end{tabular}}
\hspace{-.6cm} \\
\hline
\end{tabular}
\end{table}

\begin{table}
\caption{Dual resolution graphs and compactifying divisors for
icosahedral singularities continued.} \label{tab:icos2}
\begin{tabular}{|c|c|c|} \hline
& Dual resolution graph & Compactifying divisor \\
\hline $m=30(b-2)+23$ &
{\begin{tabular}{c}\hspace{-.5cm}\epsfig{file=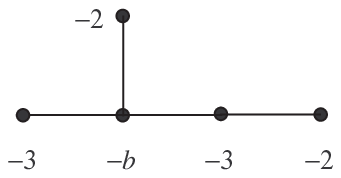}\end{tabular}}
\hspace{-.5cm} &
{\begin{tabular}{c}\hspace{-.5cm}\epsfig{file=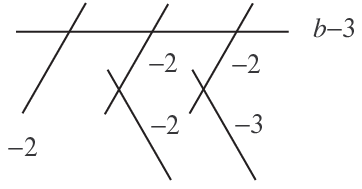}\end{tabular}}
\hspace{-.6cm} \\
\hline $m=30(b-2)+29$ &
{\begin{tabular}{c}\hspace{-.5cm}\epsfig{file=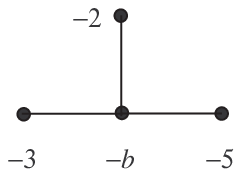}\end{tabular}}
\hspace{-.5cm} &
{\begin{tabular}{c}\hspace{-.5cm}\epsfig{file=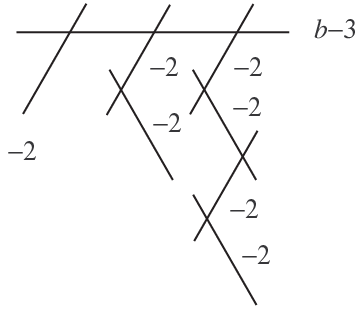}\end{tabular}}
\hspace{-.6cm} \\
\hline
\end{tabular}
\end{table}


\section{Compactification of  symplectic fillings}
\label{sec:4} Let $X$ be a symplectic filling of the link of a
quotient surface singularity.  Without loss of generality, we may
assume that $X$ is minimal, that is, $X$ does not contain any
symplectically embedded spheres of self-intersection number $-1$.
Denote by $Y$ a regular neighbourhood of the compactifying divisor
$K$ presented in Section~\ref{sec:3}. We may take $Y$ so that $Y$ is
a strong concave filling of $\partial Y \cong \partial X$ (see
\cite{Bhupal}, \cite{OO3}). Gluing $X$ and $Y$, we get a closed
symplectic manifold $Z$. The classification problem of symplectic
fillings reduces to the symplectic deformation classification of the
pair $(Z,K)$.

\subsection{Cyclic quotient singularities}
Let $L$, $C_1, \dots, C_k$ be a string of symplectically embedded
$2$-spheres in a closed symplectic $4$-manifold $M$ with $L \cdot L
=1$, $C_i \cdot C_i =v_i \ (i=1, \dots ,k)$. Note that $v_1 \leq 0$
and $v_i < 0 \ (i=2, \dots ,k)$. The main examples are the
compactifying divisors for cyclic quotient singularities given in
Section~\ref{sec:3}. Note also that $M$ is a blow-up of ${\C}P^2$ by
Theorem~\ref{thm:McDuff}.   Firstly, we show the following.

\begin{lemma} \label{lem:4.1} Let $J$ be a compatible (or tame)
almost complex structure for which $L$, $C_1 \dots ,C_k$ are
pseudoholomorphic. Then there exists at least one $J$-holomorphic
$(-1)$-curve in $M \setminus L$.
\end{lemma}

\begin{proof} If one of the $C_i \ (i\geq 2)$ is a symplectic
$(-1)$-curve, there is nothing to prove. Suppose that none of the
$C_i \ (i\geq 2)$ are symplectic $(-1)$-curves. Since $v_1 \leq 0$,
$M$ is not minimal. So there are symplectic $(-1)$-curves in $M
\setminus L$. After blowing them down we get the complex projective
plane. Denote by $E$ a symplectic $(-1)$-curve in $M \setminus L$.
If the homology class $[E]$ is represented by a $J$-holomorphic
sphere, we are done. Suppose that $[E]$ is not represented by a
$J$-holomorphic sphere. Pick a sequence of compatible almost complex
structures $J_n$ such that $L$ is $J_n$-holomorphic, $[E]$ is
represented by a $J_n$-holomorphic sphere $E_n$ and $\{J_n\}$
converges to $J$. Then $E_n$ converges to the image of a $J$-stable
map. Let $A_1, \dots ,A_l$ be its irreducible components. Since
$c_1(M)[E]=1$, there is a component $A_j$ with $c_1(M)[A_j]>0$. Note
that $A_j$ is disjoint from $L$. (Otherwise, $E \cdot L$ must be
positive.) If $A_j$ is a multiply covered component, take the
underlying reduced curve $E'$. Then Lemma~\ref{lem:2.5} implies that
$A_j$ or $E'$ is an embedded pseudoholomorphic sphere. Also,
Lemma~\ref{lem:2.8} implies that $A_j$ or $E'$ has self-intersection
$<0$. Since $c_1(M)[A_j]
>0$, by the adjunction formula, $A_j$ or $E'$ is a $J$-holomorphic $(-1)$-curve.
\qed
\end{proof}

Now we study the compactifying divisors $K$ of cyclic quotient
surface singularities. In fact, by successive blowing down, the
compactifying divisor is reduced to two complex projective lines in
the complex projective plane as Lisca claimed in \cite{L}. Here we
give a proof based on Lemma~\ref{lem:4.1} for the sake of
completeness.

We call a configuration $\cD=L \cup C_1 \cup \dots \cup C_k$ of
rational curves {\em admissible} (for symplectic fillings of links
of cyclic quotient singularities) if it is the total transform of
two distinct lines in ${\C}P^2$ under some iterated blow-up. Suppose
that $M \setminus (L \cup C_1 \cup \dots \cup C_k)$ is minimal.
Denote by ${\cal J}_\cD$ the set of compatible almost complex
structures with respect to which $\cD=L \cup C_1 \cup \dots \cup
C_k$ is pseudoholomorphic. We will blow down a maximal family
$\{E_i\}$ of pseudoholomorphic $(-1)$-curves in $M \setminus L$ to
reduce the configuration of rational curves to an admissible
configuration. Note that these $(-1)$-curves are mutually disjoint
(Lemma~\ref{lem:2.8}).

\begin{proposition} \label{prop:cycquot}
Let $J$ be a compatible almost complex structure, which is generic
in ${\cal J}_\cD$. Denote by $M'$ the symplectic $4$-manifold
obtained by blowing down all $J$-holomorphic $(-1)$-curves $\{E_j\}$
away from $L$ and by $C_i'$ the  image of $C_i$. Then $\{L, C_i'\}$
is an admissible configuration for a symplectic filling of a link of
cyclic quotient singularity.
\end{proposition}

\begin{proof} Firstly, we note that any pseudoholomorphic
$(-1)$-curve in $M'\setminus L$ is one of the $C_i' \ (i\geq 2)$.
Indeed, assume to the contrary that there is a pseudoholomorphic
$(-1)$-curve $E$ in $M'\setminus L$ which is not one of the $C_i'$.
Perturbing the almost complex structure slightly, we may assume that
$E$ does not pass through the images of the blown-down $(-1)$-curves
$E_j$. Hence we may assume that $E$ is actually a pseudoholomorphic
$(-1)$-curve already in $M\setminus L$ which contradicts the
maximality of $\{E_j\}$.

Next, we note that each $C_i'$ is embedded. It is enough to show
that each  $E_j$ intersects exactly one of $C_i$ with $E_j \cdot C_i
= 1$. Perturbing the almost complex structure away from $\cD$, we
may assume that $E_j$ intersects any $C_i$ transversally. Suppose
that $E_j$ intersects $C_i$ such that $E_j \cdot C_i \geq 2$. After
contracting $E_j$, $C_i$ becomes a nodal curve, which is singular.
The existence of such a curve is prohibited by Lemma~\ref{lem:2.5}.
Similarly, Lemma~\ref{lem:2.5} and Remark~\ref{rem:2.12} exclude the
possibility that $E_j$ intersects at least two of the $C_i$'s.

Thus Lemma~\ref{lem:4.1} implies that one of the $C_i'$ is a
symplectic $(-1)$-curve. After blowing it down, we get a new
configuration of rational curves $L \cup C_1'' \cup \dots \cup
C_{k-1}''$. Let $M''$ denote the resulting ambient symplectic
4-manifold.

We claim that there are no symplectic $(-1)$-curves in $M''
\setminus L$ except for some of the $C_i''$. Suppose that $E$ is
such a pseudoholomorphic $(-1)$-curve. A similar argument as above
shows that $E$ can be lifted to a pseudoholomorphic $(-1)$-curve in
$M\setminus L$. This contradicts the maximality of $\{E_j\}$.

Continuing this process, we can successively blow down $(-1)$-curves
to get the complex projective plane and $C_1$ transformed to a
complex projective line $L'$ transversal to $L$. The other $C_i$ are
contracted to a point on $L'$. \qed
\end{proof}

\subsection{Dihedral singularities}
Let $\widetilde M$ be a closed symplectic 4-manifold containing a
configuration of symplectically embedded 2-spheres intersecting in
the manner shown in the first picture in Figure~\ref{fig:dihedbl}.
Here $c,c_i\geq 2, i=1,\ldots, k$. By a sequence of blow-downs and a
blow-up, as in \cite{OO2} in the simple dihedral case, we can
transform this configuration into a configuration containing a cusp
curve (see Figure~\ref{fig:dihedbl}). Let $M$ denote the resulting
ambient manifold.

\begin{figure}
\begin{center}
\includegraphics[angle=90]{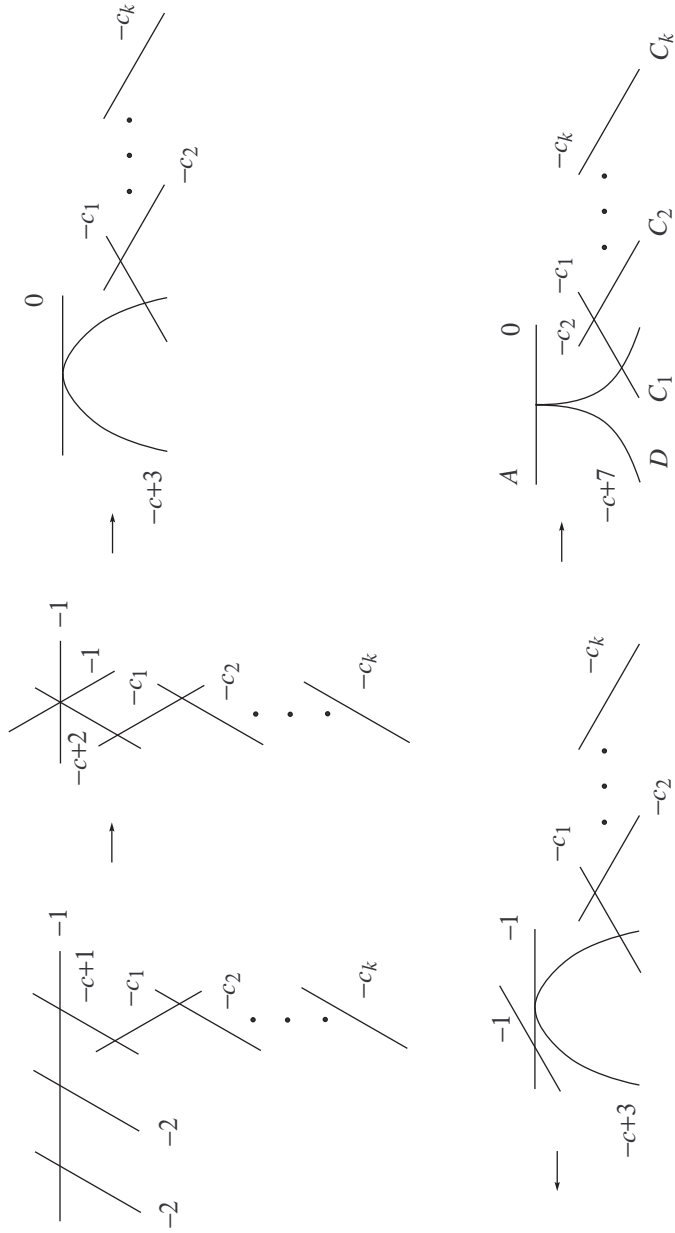}
\end{center}
\caption{Sequence of blow-downs and a blow-up transforming the
compactifying divisor of a dihedral singularity into a configuration
containing a cusp curve} \label{fig:dihedbl}
\end{figure}

To obtain a classification of fillings of links of dihedral
singularities we proceed via a process of blowing down symplectic
$(-1)$-curves in $M$. Using the results of \cite{OO2}, we know that
the complement of a regular neighbourhood of the union of the
0-curve $A$ and the cusp curve $D$, in the final picture in
Figure~\ref{fig:dihedbl}, gives a symplectic filling of the link of
the simple dihedral singularity $D_{c+1}$. When $c=2$, $D_3=A_3$ and
the argument for simple dihedral singularities also works for the
$A_3$-singularity. It follows that $M$, and hence $\widetilde M$, is
in fact rational.

Consider now, generally, a closed symplectic 4-manifold $Z$
containing a configuration of rational curves $\cD=A\cup D\cup
C_1\cup\cdots\cup C_k$ as depicted in Figure~\ref{fig:dihedconfig}.
Here $D$ is a singular curve with a $(2,3)$-cusp, $A$ is an embedded
$0$-curve intersecting $D$ at the cusp point and $C_1,\ldots C_k$
are embedded curves. By \cite{OO3}, $D\cdot D\leq 8$. Also, by
Lemma~\ref{lem:2.9} and Lemma~\ref{lem:2.10}, $C_i\cdot C_i\leq -1$
for all $i$.

\begin{figure}
\begin{center} \epsfig{file=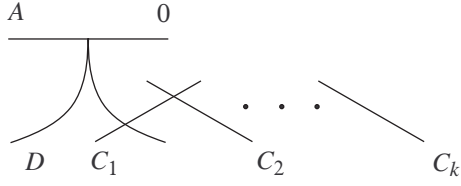} \end{center}
\caption{General configuration of rational curves considered in case
of symplectic fillings of links of dihedral singularities}
\label{fig:dihedconfig}
\end{figure}

\begin{lemma} \label{lem:dihed} Assume that the string $C_1,\ldots,C_k$
is nonempty. Let $J$ be a compatible almost complex structure for
which $A,D,C_1,\ldots,C_k$ are pseudoholomorphic. Then $C_1$ is a
$(-1)$-curve or there exists a $J$-holomorphic $(-1)$-curve in
$Z\setminus (A\cup D)$.
\end{lemma}

\begin{proof} First suppose that the complement of $A\cup D$ is minimal. If
$D\cdot D\leq 5$, then we know from \cite{OO2} that the
anti-canonical class of $Z$ is represented by $D$. Since $C_1\cdot
D=1$ it follows that $C_1$ is a $(-1)$-curve. If $D\cdot D\geq 6$,
then, after blowing up 3 points on the smooth part of $D$, away from
$C_1$, denote the proper transform of $D$ by $\widetilde D$, the
blow up of $Z$ by $\widetilde Z$ and the exceptional curves by
$e_1,e_2,e_3$. Then $\widetilde D\cdot\widetilde D\leq 5$. If the
complement of $A\cup\widetilde D$ is minimal then $C_1$ is a
$(-1)$-curve and we are done. If there exists a $(-1)$-curve $E$ in
$\widetilde Z\setminus (A\cup \widetilde D)$, then we claim that $E$
is disjoint from the exceptional curves $e_1,e_2,e_3$. To see this,
note that, by Proposition~\ref{prop:2.4}, we may assume that
$A,\widetilde D,e_1,e_2,e_3,E$ are pseudoholomorphic with respect to
some compatible almost complex structure $\widetilde J$. Suppose
that $E\cdot e_i\geq 1$ for some $i$. Then, after blowing down $E$,
the image of $e_i$ is either a singular pseudoholomorphic curve
which intersects $\widetilde D$ once only, contradicting
Lemma~\ref{lem:2.7}, or is an embedded pseudoholomorphic 0-curve
intersecting $\widetilde D$ once only, contradicting
Lemma~\ref{lem:2.10}. Thus $E$ exists as a symplectic $(-1)$-curve
in $Z\setminus (A\cup D)$, contradicting the minimality of the
latter. Thus this case does not occur.

Suppose now that $Z\setminus (A\cup D)$ is not minimal. Then there
exists a symplectic $(-1)$-curve $E$ in $Z\setminus (A\cup D)$. If
$[E]$ is represented by a $J$-holomorphic curve, we are done. If
$[E]$ can not be represented by a $J$-holomorphic curve, then pick a
sequence of tame almost complex structures $J_n$ converging to $J$
such that $[E]$ is represented by a $J_n$-holomorphic curve $E_n$
for each $n$. By the compactness theorem, after taking a
subsequence, $E_n$ converges to the image of a stable map. Let
$A_1,\ldots,A_l$ denote the irreducible components of this stable
map. Arguing as in the proof of Proposition 4.1 in \cite{OO3}, it
can be shown that none of the $A_i$'s coincides with (or multiply
covers) $D$.  Hence it also follows that none of the $A_i$'s
intersects $D$. Note also that $A_j$ is disjoint from $A$.  Since
$c_1(Z)[E]=1$, it follows that $c_1(Z)[A_j]>0$ for some $j$. By
replacing $A_j$ by the underlying simple curve, if $A_j$ is
multiply-covered, assume that $A_j$ is a simple curve. By
Lemma~\ref{lem:2.6}, $A_j$ is an embedded $J$-holomorphic sphere.
Also, by Lemma~\ref{lem:2.9}, $A_j\cdot A_j<0$. By the adjunction
formula, it now follows that $A_j$ is a $J$-holomorphic
$(-1)$-curve. \qed
\end{proof}

In general, we call a configuration of rational curves $\cD=A\cup
D\cup C_1\cup\cdots\cup C_k$ as in Figure~\ref{fig:dihedconfig}, in
a closed symplectic 4-manifold $Z$, {\em admissible} (for symplectic
fillings of links of dihedral singularities) if it can be obtained
as the total transform of an iterated blow-up of either:
\begin{enumerate}[\rm (a)]
\item a union of a cuspidal rational curve of bi-degree $(2,2)$ and
a 0-curve intersecting only at the cusp point in the ruled surface
$\C P^1\times\C P^1$, or
\item a union of a singular rational
curve representing the class $3\C P^1-L$ and a 0-curve intersecting
only at the cusp point in the one-point blow-up of $\C P^2$, where
$L$ represents the exceptional curve of the blow-up.
\end{enumerate}

We call $A\cup D\cup C_1\cup\cdots\cup C_k$ a {\em pre-admissible}
configuration if it becomes admissible after possibly blowing down
some $(-1)$-curves intersecting only $D$.

Assume now that $M\setminus (A\cup D\cup C_1\cup\cdots\cup C_k)$ is
minimal. The following proposition shows that after blowing down a
maximal family of pseudoholomorphic $(-1)$-curves in $M\setminus
(A\cup D)$ the configuration $A\cup D\cup C_1\cup\cdots\cup C_k$ is
reduced to a pre-admissible configuration. Note that by
Lemma~\ref{lem:2.11} these $(-1)$-curves are necessarily disjoint.
Note also that these $(-1)$-curves, if they are not contained in the
string $C_1,\ldots,C_k$, can intersect it at most once, that is, for
any such $(-1)$-curve $E$, $\sum E\cdot C_i\leq 1$. Indeed, suppose
that there is a $(-1)$-curve $E$ such that $\sum E\cdot C_i> 1$,
then after contracting $E$ we get either a singular
pseudoholomorphic curve or a cycle of pseudoholomorphic spheres
whose intersection number with $A$ is $0$.  This contradicts
Lemma~\ref{lem:2.9} and Remark~\ref{rem:2.12}. Now let $\cJ_\cC$
denote the set of compatible almost complex structures with respect
to which $\cC=A\cup D\cup C_1\cup\cdots\cup C_k$ is
pseudoholomorphic.

\begin{proposition} \label{prop:dihed} Let $J$ be a compatible almost complex structure
which is generic in $\cJ_\cC$. Denote by $M^{\pr}$ the symplectic
$4$-manifold obtained by blowing down all $J$-holomorphic
$(-1)$-curves in $M\setminus (A\cup D)$ and by $C_i^\pr$ the image
of $C_i$. Then $\{A,D,C_i^\pr\}$ is a pre-admissible configuration
for a symplectic fillings of a link of a dihedral singularity.
\end{proposition}

\begin{proof} By Lemma~\ref{lem:dihed} there exists a $(-1)$-curve in
$M^\pr\setminus (A\cup D)$. We claim that this $(-1)$-curve must be
one of the $C_i^\pr$. Suppose that it is not. Then it can intersect
only one of the $C_i^\pr$ at exactly one point transversally. Hence
it must already have been in $M\setminus (A\cup D)$ contradicting
the fact that we blew down all such $(-1)$-curves. We now blow down
this curve to obtain a new configuration $\{A,D^\pr,C_i^{\pr\pr}\}$.
If the string $\{C_i^{\pr\pr}\}$ is not empty, we can argue in a
similar way to show that it must also contain a $(-1)$-curve.
Blowing down this $(-1)$-curve also and continuing in this way we
can show that the whole string $\{C_i^\pr\}$ can eventually be blown
down. Thus we get the compactification of the symplectic filling of
the link of a simple dihedral singularity or the $A_3$-singularity
in \cite{OO2}, after blowing up, if necessary, at most 3 points on
the smooth part of the image of $D$. (In \cite{OO2}, we dealt with
the $A_3$-singularity as one of the $A_n$-singularities not as the
``$D_3$-singularity''. However, the argument for the dihedral case
does work for the $D_3=A_3$-case.) In fact, we can see that there
are pseudoholomorphic $(-1)$-curves, which intersect $D$ once
transversally. After contracting them, we get the configuration of
type (a) or (b). As we have noticed before, these $(-1)$-curves
eventually exist in the original manifold $M$. Changing the order of
the blowing down processes, we find that $\{A,D,C_i^\pr\}$ is a
pre-admissible configuration. \qed
\end{proof}

\subsection{Tetrahedral, octahedral and icosahedral singularities}
For the purposes of classification of symplectic fillings of
tetrahedral, octahedral and icosahedral singularities, it is
convenient to divide these singularities into two sets: those with a
branch consisting of two $(-2)$-curves intersecting the central
curve in the minimal resolution and those with a branch consisting
of a single $(-3)$-curve intersecting the central curve. We
designate these singularities as type $(3,2)$ singularities and type
$(3,1)$ singularities respectively. There is exactly one class of
quotient singularities which lies in both sets, namely the
tetrahedral singularities $T_{6(b-2)+3}$.

We begin by discussing the classification of fillings of type
$(3,2)$ singularities. Note that, given a singularity $\G$ of type
$(3,2)$, we can always choose a compactifying divisor such that its
central curve has self-intersection number $-1$. Namely, if $b=2$,
the central curve given in Section 2 is a $(-1)$-curve. If $b \geq
3$, we blow up the compactifying divisor given earlier at the
transversal intersection of the central curve and the third branch,
that is, the branch that does not consist of a single $(-2)$- or
$(-3)$-curve (except in case $\G=T_{6(b-2)+1}$, in which case we
blow up at the intersection of the central curve and one of the
$(-3)$-curves), repeatedly until the self-intersection number of the
central curve has dropped to $-1$. Now let $\widetilde M$ be a
closed symplectic 4-manifold containing a configuration of
symplectically embedded 2-spheres intersecting in the manner shown
in the first picture in Figure~\ref{fig:2-2bl}. Here $a\geq 1$ and
$c_i\geq 2$ for $i=1,\ldots,k$. Note that $a=1$, when $b\geq 3$.
When $b = 2$, $2 \leq a \leq 5$ (see Tables 1, 2, 3). It will be
convenient to transform this configuration of symplectically
embedded 2-spheres into one containing a cusp curve. We achieve this
by a sequence of blow-downs, as in the cases $E_6,E_7,E_8$ in
\cite{OO2} (see Figure~\ref{fig:2-2bl}).
\begin{figure}
\begin{center}
\includegraphics[angle=90]{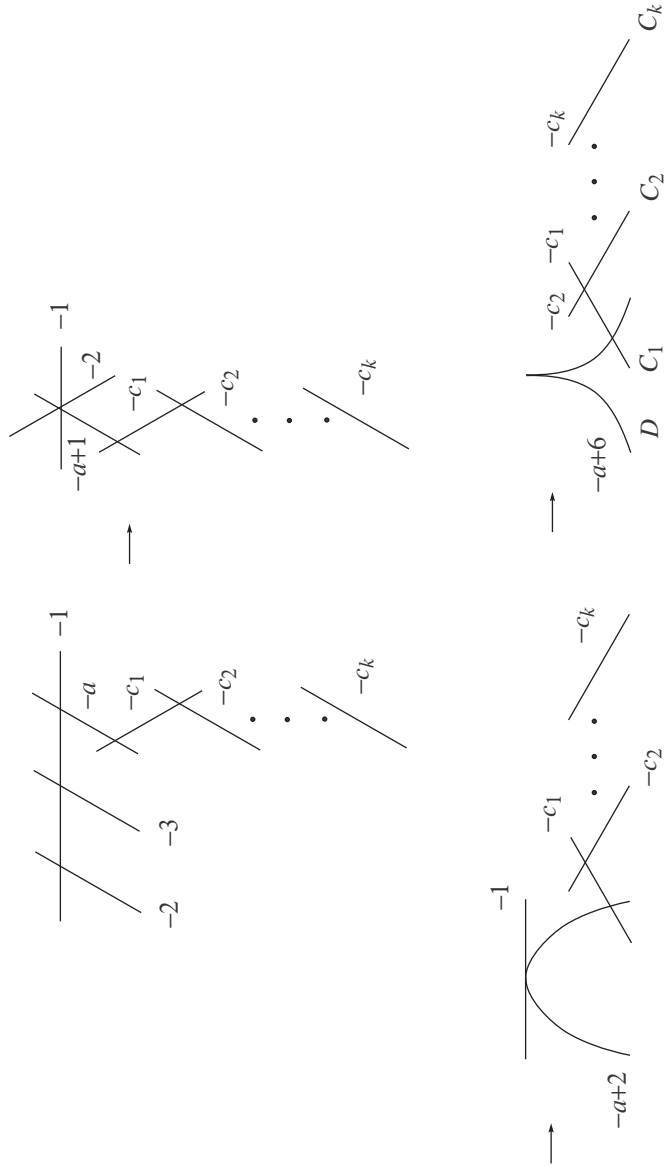}
\end{center}
\caption{Sequence of blow-downs transforming the compactifying
divisor of a tetrahedral, octahedral or icosahedral singularity of
type $(3,2)$ into a configuration containing a cusp curve}
\label{fig:2-2bl}
\end{figure}
Let $M$ denote the resulting closed symplectic 4-manifold, $D$ the
cusp curve and $C_1,\ldots,C_k$ the string of curves attached to
$D$. As the self-intersection number of the cusp curve $D$ is always
positive, we can immediately conclude, by the main theorem in
\cite{OO3}, that $M$, and hence $\widetilde M$, is a rational
symplectic manifold.

Consider now, generally, a closed symplectic 4-manifold $Z$
containing a configuration of rational curves $\cD=D\cup
C_1\cup\cdots\cup C_k$ as depicted in Figure~\ref{fig:2-2blconfig}.
Here $D$ is a singular curve with a $(2,3)$-cusp and $C_1,\ldots
C_k$ are embedded curves. By \cite{OO3}, $D\cdot D\leq 9$. Also, by
Lemma~\ref{lem:2.9} and Lemma~\ref{lem:2.10}, $C_i\cdot C_i\leq -1$
for all $i$.

\begin{figure}
\begin{center} \epsfig{file=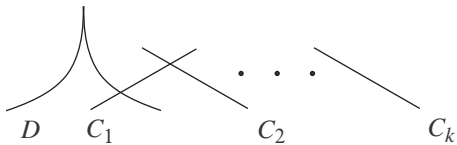} \end{center}
\caption{General configuration of rational curves considered in case
of symplectic fillings of links of tetrahedral, octahedral and
icosahedral singularities of type $(3,2)$} \label{fig:2-2blconfig}
\end{figure}

\begin{lemma} Assume that the string $C_1,\ldots,C_k$ is nonempty.
Let $J$ be a compatible almost complex structure for which
$D,C_1,\ldots,C_k$ are pseudoholomorphic. Then $C_1$ is a
$(-1)$-curve or there exist a $J$-holomorphic $(-1)$-curve in
$Z\setminus D$.
\end{lemma}

\begin{proof} Suppose that the complement of $D$ is minimal and $D\cdot
D\leq 5$. Then, from \cite{OO2}, we know that an anti-canonical
divisor is given by $D$. It follows that $C_1$ is a $(-1)$-curve.
The remainder of the proof proceeds in a similar way to the proof of
Lemma~\ref{lem:dihed}. \qed
\end{proof}

In general, we call a configuration of rational curves $\cD=D\cup
C_1\cup\cdots\cup C_k$ as in Figure~\ref{fig:2-2blconfig}, in a
closed symplectic 4-manifold $Z$, {\em admissible} (for symplectic
fillings of links of tetrahedral, octahedral and icosahedral
singularities of type $(3,2)$) if it can be obtained as the total
transform of an iterated blow-up of the cuspidal cubic curve in $\C
P^2$, or of the cuspidal curve of bi-degree $(2,2)$ in $\C P^1
\times \C P^1$. In \cite{OO2}, we can always blow down a maximal
family of $(-1)$-curves to get a cuspidal cubic curve in $\C P^2$.
In fact, if $b_2(M) \geq 3$, one can also contract another maximal
family of $(-1)$-curves to get $\C P^1 \times \C P^1$. In our
present situation, we do not blow down $(-1)$-curves intersecting
both $D$ and some $C_i$.  Hence we may not arrive at $\C P^2$ but at
$\C P^1 \times \C P^1$. Again, we call such a configuration {\em
pre-admissible} if it becomes admissible after possibly blowing down
some $(-1)$-curves intersecting only $D$.

Now assume that $M\setminus (D\cup C_1\cup\cdots\cup C_k)$ is
minimal. The following proposition shows that after blowing down a
maximal family of pseudoholomorphic $(-1)$-curves in $M\setminus D$
the configuration $D\cup C_1\cup\cdots\cup C_k$ is reduced to a
pre-admissible configuration. By Lemma~\ref{lem:2.11}, these
$(-1)$-curves are necessarily disjoint. Also, these $(-1)$-curves,
if they are not contained in the string $C_1,\ldots,C_k$, can
intersect it at most once. Indeed, suppose that there is such a
$(-1)$-curve $E$ such that $\sum E\cdot C_i\geq 2$, then, after
contracting $E$, the image of all the curves $C_i$ contains a
singular pseudoholomorphic curve or a cycle of pseudoholomorphic
spheres. This contradicts Lemma~\ref{lem:2.6}, Lemma~\ref{lem:2.7}
or Remark~\ref{rem:2.12}, thus proving the assertion. (Note that for
type $(3,2)$ singularities $\G$ for which the string
$C_1,\ldots,C_k$ is nonempty, the self-intersection number $D\cdot
D\geq 3$.) In summary, after contracting such a pseudoholomorphic
$(-1)$-curve, we again get a configuration consisting of $D$ and a
string of embedded spheres. We can prove the following proposition
in a similar way to the proof Proposition~\ref{prop:dihed}. Let
$\cJ_\cC$ denote the set of compatible almost complex structures
with respect to which $\cC=D\cup C_1\cup\cdots\cup C_k$ is
pseudoholomorphic.

\begin{proposition} Let $J$ be a compatible almost complex structure
which is generic in $\cJ_\cC$. Denote by $M^{\pr}$ the symplectic
$4$-manifold obtained by blowing down all $J$-holomorphic
$(-1)$-curves in $M\setminus \cD$ and by $C_i^\pr$ the image of
$C_i$. Then $\{D,C_i^\pr\}$ is a pre-admissible configuration for a
symplectic filling of a link of a tetrahedral, octahedral or
icosahedral singularity of type $(3,2)$.
\end{proposition}

We now turn to the classification of fillings of type $(3,1)$
singularities. Again, given a singularity $\G$ of type $(3,1)$, as
in the case of type $(3,2)$ singularities, we can always choose a
compactifying divisor whose central curve has self-intersection
number $-1$. Namely, as before, when $b \geq 3$, we blow up the
compactifying divisor given earlier at the transversal intersection
of the central curve and the third branch, that is, the branch that
does not consist of one or two $(-2)$-curves (except in case
$\G=T_{6(b-2)+5}$, in which case we blow up at the intersection of
the central curve and one of the branches consisting of 2
$(-2)$-curves), repeatedly until the self-intersection number of the
central curve has dropped to $-1$. Now let $\widetilde M$ be a
closed symplectic 4-manifold containing a configuration of
symplectically embedded 2-spheres intersecting in the manner shown
in the first picture in Figure~\ref{fig:3-bl}. Here $a\geq 1$ and
$c_i\geq 2$ for $i=1,\ldots,k$. Again it will be convenient to
transform this configuration of symplectically embedded 2-spheres
into one containing a cusp curve. We achieve this by a sequence of
blow-downs and a blow-up (see Figure~\ref{fig:3-bl}).
\begin{figure}
\begin{center}
\includegraphics[angle=90]{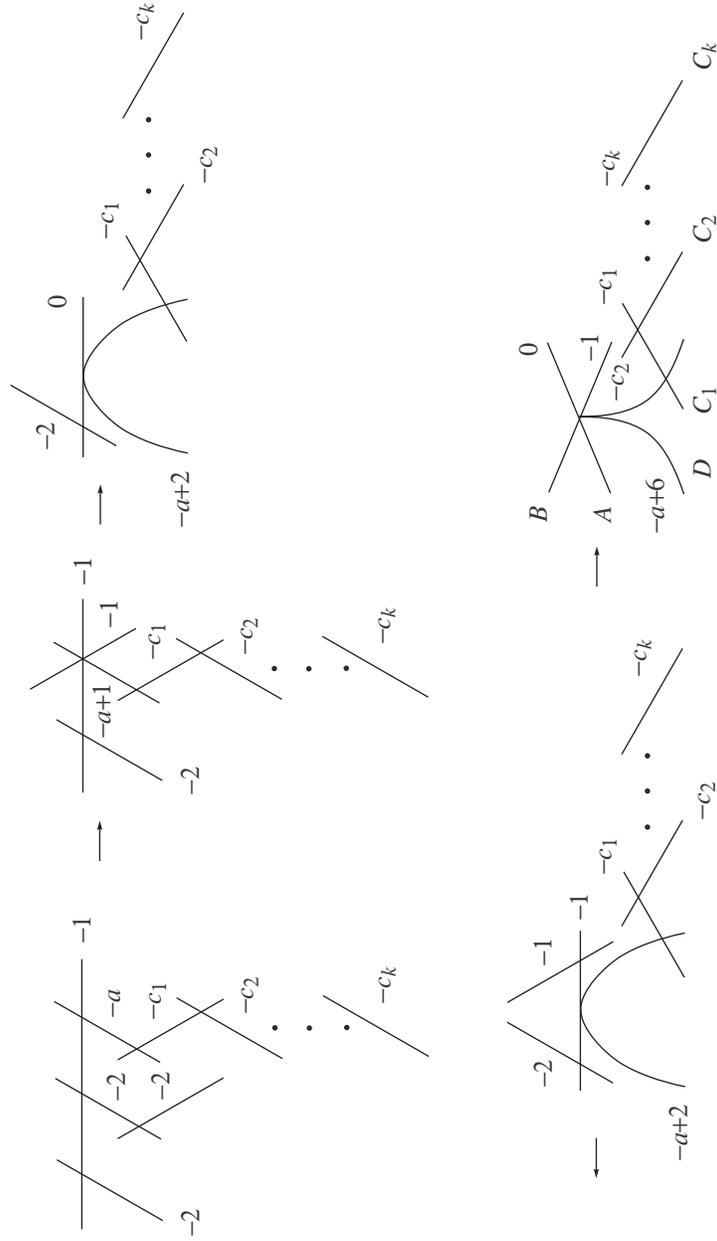}
\end{center}
\caption{Sequence of blow-downs and a blow-up transforming the
compactifying divisor of a tetrahedral, octahedral or icosahedral
singularity of type $(3,1)$ into a configuration containing a cusp
curve} \label{fig:3-bl}
\end{figure}
Let $M$ denote the resulting closed symplectic 4-manifold, $D$ the
cuspidal curve, $A$ the 0-curve, $B$ the $(-1)$-curve and
$C_1,\ldots,C_k$ the string of curves intersecting $D$. Since the
complement of a regular neighbourhood of $A\cup D$ is a symplectic
filling of a simple dihedral singularity or $A_3$, it follows, by
the results of \cite {OO2}, that $M$, and hence $\widetilde M$, is a
rational symplectic 4-manifold. Note that $A_3$ can be also treated
as a simple dihedral singularity ``$D_3$''. Let $\cJ_{\cC}$ denote
the set of compatible almost complex structures with respect to
which $\cC=A\cup B\cup D\cup C_1\cup\cdots\cup C_k$ is
pseudoholomorphic.

\begin{lemma} \label{lem:type31}
Let $J$ be generic in $\cJ_{\cC}$. Then there exists a
$J$-holomorphic $(-1)$-curve $E$ in $M\setminus (A\cup D)$ such that
$B\cdot E=1$. Moreover, such a curve $E$ is unique.
\end{lemma}

\begin{proof} Let $\cJ_{\cD}$ denote the set of compatible almost complex
structures with respect to which $\cD=A\cup B\cup D$ is
pseudoholomorphic. By Proposition~\ref{prop:2.4} applied to the
curves $A\cup B\cup D$, for a generic almost complex structure
$J^\pr\in\cJ_{\cD}$, any symplectic $(-1)$-curve in $M$ has a unique
$J^\pr$-holomorphic representative. In particular, for any maximal
disjoint family $E_1,\ldots,E_N$ of symplectic $(-1)$-curves in the
complement of $A\cup D$, we have a unique family of
$J^\pr$-holomorphic representatives $E_1^\pr,\ldots,E_N^\pr$. It
follows, by the results of \cite {OO2}, that an anti-canonical
divisor of $M$ is given by $D-\sum E_i^\pr$. Now since $c_1(M)[B]=1$
and $B\cdot D=2$, it follows that there is exactly one $E_i^\pr$
such that $B\cdot E_i^\pr=1$ and $B\cdot E_j^\pr=0$ if $j\neq i$.
Assume, without loss of generality, that $B\cdot E_1^\pr=1$. Now
consider a sequence of generic almost complex structures
$J_n\in\cJ_{\cD}$ converging to $J$ such that $[E_1^\pr]$ is
represented by a $J_n$-holomorphic curve $B_n$ for each $n$. By the
compactness theorem, $B_n$ converges to the image of a stable map.
Let $A_1,\ldots,A_l$ denote the irreducible components of this
stable map. By the proof of Proposition~4.1 in \cite{OO3}, no
component of this stable map coincides with (or is a multiple cover
of) $D$. Hence, in particular, no component can intersect $D$. Since
$[B]\cdot [E_1^\pr]=1$, there is a component $A_i$ such that
$A_i\cdot B=1$. It follows that $A_i$ is a simple curve. Since $A_i$
is disjoint from $D$, by Lemma~\ref{lem:2.9}, $A_i$ is rational
curve of negative self intersection. Now the fact that $J$ is
generic away from the configuration $\cC$ allows us to conclude that
$A_i$ is in fact a $(-1)$-curve. (The virtual dimension of the
moduli space of singular pseudoholomorphic curves of negative
self-intersection number is negative.) Taking $E=A_i$ gives the
required $J$-holomorphic $(-1)$-curve.

If $E$ and $E'$ are pseudoholomorphic $(-1)$-curves such that $E
\cdot B=E' \cdot B=1$, then, for a generic $J^\pr\in\cJ_{\cD}$, both
$[E]$ and $[E']$ are represented by $J^{\pr}$-holomorphic
$(-1)$-curves. By Lemma~\ref{lem:2.11}, they are mutually disjoint
and we may assume that $E$ and $E'$ are contained in $\{
E_1^\pr,\ldots,E_N^\pr \}$. However, as we saw, there is exactly one
$E_i^\pr$ such that $B\cdot E_i^\pr=1$.  Hence we find uniqueness,
that is, $E=E'$. \qed
\end{proof}

Note that $E$ can intersect at most one of the $C_i$'s (see
Remark~\ref{rem:2.12}). In a similar way, $E \cdot C_i=1$ if $E$
intersects $C_i$. There are now two cases to consider: the case
where $E$ is disjoint from the string $C_1,\ldots, C_k$ and the case
where $E$ intersects precisely one member of the string
$C_1,\ldots,C_k$.

\subsubsection*{{\it Case I: $E\cdot C_i=0$ for all $i$}} In this case,
blow down the $(-1)$-curve $E$ and denote the image of $B$ under the
blowing down map by $B^\pr$. Then $B^\pr$ is a 0-curve and the
resulting configuration $\cC^{\pr}=A\cup B^\pr\cup D\cup
C_1\cup\cdots\cup C_k$ is as in Figure~\ref{fig:case1}.
\begin{figure}
\begin{center}\epsfig{file=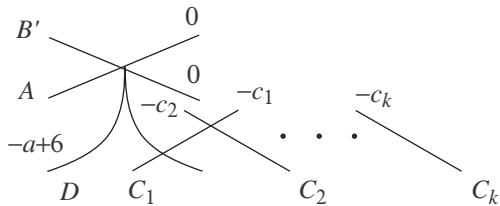}\end{center}
\caption{Image of compactifying divisor after blowing down the
$(-1)$-curve $E$ in Case~I} \label{fig:case1}
\end{figure}
Let $M^{\pr}$ denote the resulting symplectic 4-manifold. We will
show that, after blowing down $(-1)$-curves in $M^{\pr}\setminus
(A\cup B^\pr\cup D)$, the string $C_1,\ldots,C_k$ is transformed
into one which can be sequentially blown down.

Consider, generally, a closed symplectic 4-manifold $Z$ containing a
configuration of rational curves $\cD=A\cup B\cup D\cup
C_1\cup\cdots\cup C_k$ as depicted in Figure~\ref{fig:case1config}.
Here $D$ is a singular curve with a $(2,3)$-cusp, $A$ and $B$ are
embedded 0-curves intersecting transversely at the cusp point of $D$
and $C_1,\ldots,C_k$ are embedded curves. By \cite{OO3}, $D\cdot
D\leq 8$. Also, by Lemma~\ref{lem:2.9} and Lemma~\ref{lem:2.10},
$C_i\cdot C_i\leq -1$ for all $i$.

\begin{figure}
\begin{center} \epsfig{file=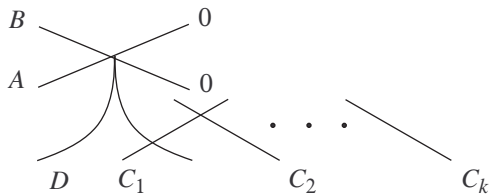} \end{center}
\caption{General configuration of rational curves considered in case
of case~I symplectic fillings of links of tetrahedral,
octahedral and icosahedral singularities of type $(3,1)$}
\label{fig:case1config}
\end{figure}

\begin{lemma} \label{lem:case1} Assume that the string
$C_1,\ldots,C_k$ is nonempty. Let $J$ be a compatible almost complex
structure for which $A,B,D,C_1,\ldots,C_k$ are pseudoholomorphic.
Then $C_1$ is a $(-1)$-curve or there exist a $J$-holomorphic
$(-1)$-curve in $Z\setminus (A\cup B\cup D)$.
\end{lemma}

\begin{proof} Suppose that the complement of $A\cup B\cup D$ is minimal and
$D\cdot D\leq 5$, then it follows from \cite{OO2} that an
anti-canonical divisor of $Z$ is given by $D$. It follows that $C_1$
is a $(-1)$-curve. The remainder of the proof proceeds in a similar
way to proof of Lemma~\ref{lem:dihed}. \qed
\end{proof}

In general, we call a configuration of rational curves $\cD=A\cup
B\cup D\cup C_1\cup\cdots\cup C_k$ as in
Figure~\ref{fig:case1config}, in a closed symplectic 4-manifold $Z$,
{\em admissible} (for case~I symplectic fillings of links of
tetrahedral, octahedral and icosahedral singularities of type
$(3,1)$) if it can be obtained as the total transform of an iterated
blow-up of a union of a cuspidal rational curve of bi-degree $(2,2)$
and two 0-curves intersecting transversely at its cusp point in $\C
P^1\times\C P^1$. Again, we call such a configuration {\em
pre-admissible} if it becomes admissible after possibly blowing down
some $(-1)$-curves intersecting $D$ only.

Assume now that $M^{\pr}\setminus (A\cup B^\pr\cup D\cup
C_1\cup\cdots\cup C_k)$ is minimal. The following proposition shows
that after blowing down a maximal family of $(-1)$-curves in
$M^{\pr}\setminus (A\cup B^\pr\cup D)$ the configuration
$\cC^\pr=A\cup B^\pr\cup D\cup C_1\cup\cdots\cup C_k$ is reduced to
a pre-admissible configuration. Note that, by construction, $C_2,
\dots,C_k$ are not $(-1)$-curves and that $C_1$ intersects $D$. By
Lemma~\ref{lem:2.9}, these $(-1)$-curves are necessarily disjoint.
Again, by Lemma~\ref{lem:2.6}, Lemma~\ref{lem:2.7} and
Remark~\ref{rem:2.12}, any such $(-1)$-curve, if it is not contained
in the string $C_1,\ldots,C_k$, can intersect it at most once. Let
$\cJ_{\cC^{\pr}}$ denote the set of compatible almost complex
structures with respect to which all the irreducible components of
the configuration $\cC^{\pr}$ are pseudoholomorphic.

\begin{proposition} Let $J$ be a compatible almost complex structure
which is generic in $\cJ_{\cC^{\pr}}$. Denote by $M^{\pr\pr}$ the
symplectic $4$-manifold obtained by blowing down all $J$-holomorphic
$(-1)$-curves in $M^{\pr}\setminus (A\cup B^\pr\cup D)$ and by
$C_i^\pr$ the image of $C_i$. Then $\{A,B^\pr,D,C_i^\pr\}$ is a
pre-admissible configuration for a case~I symplectic filling of a
link of a tetrahedral, octahedral or icosahedral singularity of type
$(3,1)$.
\end{proposition}

\subsubsection*{{\it Case II: $E\cdot C_i=1$ for some $i$}} Again,
begin by blowing down $E$. Denote the resulting symplectic
4-manifold by $M^{\pr}$, the image of $B$ by $B^\pr$ and the image
of $C_j$ by $C_j^\pr$ for $j=1,\ldots,k$. Then $B^\pr$ is a 0-curve,
$C_i^\pr\cdot C_i^\pr=-c_i+1$ and the resulting configuration
$\cC^{\pr}=A\cup B^\pr\cup D\cup C_1^\pr\cup\cdots\cup C_k^\pr$ is
as in Figure~\ref{fig:case2}. As in other cases, we blow down some
pseudoholomorphic $(-1)$-curves and reduce the compactifying divisor
to a standard form. For each pseudoholomorphic $(-1)$-curve $F$ in
the complement of $A \cup B'$, $F$ intersects at most one of the
$C_j^\pr$.  Moreover their intersection number is $1$. After
contracting such $(-1)$-curves, $C_1^\pr \dots , C_k^\pr$ remain
symplectically embedded spheres, whose dual graph is a string. The
image $C_j^{\pr \pr}$ of $C_j^\pr$ ($j \neq i$) is contained in the
complement of $A \cup B'$, hence has negative self-intersection
number. Note that $F$ may also intersect $D$. In such a case, we
have $F \cdot D =1$.  If $F \cdot D > 1$, after contracting $F$, the
image of $D$ contains at least two singular points.  However, the
intersection with $A$ (resp.\ $B'$) remains the same.  Thus, after
contracting other $(-1)$-curves in the complement of $A \cup B'$,
the image of $D$ should represent a homology class of bi-degree
$(2,2)$. This is absurd.

\begin{figure}
\begin{center}\epsfig{file=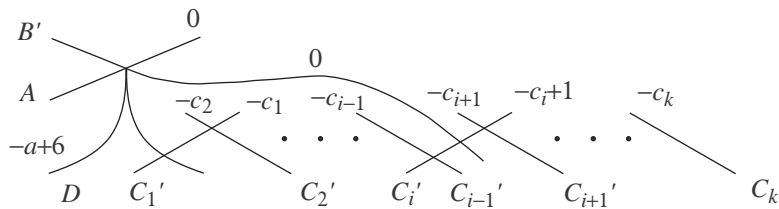}\end{center}
\caption{Image of compactifying divisor after blowing down the
$(-1)$-curve $E$ in Case~II} \label{fig:case2}
\end{figure}

Consider now, generally, a closed symplectic 4-manifold $Z$
containing a configuration of rational curves $\cD=A\cup B\cup D\cup
C_1\cup\cdots\cup C_k$ intersecting each other as in
Figure~\ref{fig:case2config}, however with the possibility that
there might be more intersections between the string $C_1, \dots,
C_k$ and the singular curve $D$ than indicated in the figure. Here
$D$ is a singular curve with a $(2,3)$-cusp, $A$ and $B$ are
embedded 0-curves intersecting transversely at the cusp point of $D$
and $C_1,\ldots,C_k$ are embedded curves. By \cite{OO3}, $D\cdot
D\leq 8$. Also, by Lemma~\ref{lem:2.9} and Lemma~\ref{lem:2.10},
$C_j\cdot C_j\leq -1$ for all $j$.

\begin{figure}
\begin{center}\epsfig{file=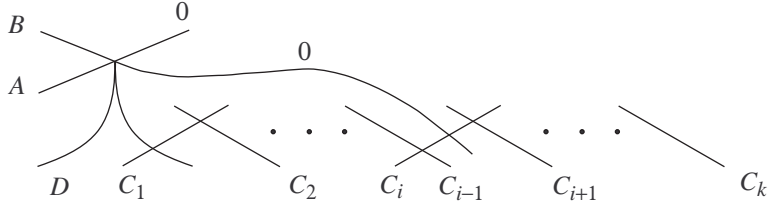}\end{center}
\caption{General configuration of rational curves considered in case
of case~II symplectic fillings of links of tetrahedral,
octahedral and icosahedral singularities of type $(3,1)$}
\label{fig:case2config}
\end{figure}

\begin{lemma} \label{lem:case2}
Let $J$ be a compatible almost complex structure for which the
irreducible components of the configuration $\cD$ are
pseudoholomorphic. Then there exist a $J$-holomorphic $(-1)$-curve
in $Z\setminus (A\cup B)$, unless $k=1$, $C_1 \cdot C_1 = 0$ and $D
\cdot D = 8$.
\end{lemma}

\begin{proof} By \cite{OO2}, we know that after collapsing a maximal family
of $(-1)$-curves in the complement of $A\cup B$ the manifold $Z$ is
reduced to $\C P^1\times\C P^1$ with the image $D$ being a
pseudoholomorphic cuspidal rational curve of bi-degree $(2,2)$. If
$k >1$, there are at least two $C_j$, one of which is contained in
the complement of $A \cup B$. Hence its self-intersection number is
negative. Suppose that $C_1$ is the only member of the string and
$C_1 \cdot C_1 \neq 0$. Since $C_1$ does not intersect $A$, the
self-intersection number of $C_1$ must be negative. Otherwise the
self-intersection number of $D$ is less than 8. In each case, we
find that the complement of $A \cup B$ is not minimal. The proof now
proceeds in a similar way to the proof of Lemma~\ref{lem:dihed}.
\qed
\end{proof}

In general, we say that a configuration of rational curves
$\cD=A\cup B\cup D\cup C_1\cup\cdots\cup C_k$ in a closed symplectic
4-manifold $Z$, where $A\cup B\cup D$ are as in
Figure~\ref{fig:case2config} and $C_1, \dots, C_k$ is a string of
embedded curves, is {\em admissible} (for case~II symplectic
fillings of links of tetrahedral, octahedral or icosahedral
singularities of type $(3,1)$) if it can be obtained as a total
transform, under an iterated blow-up, of a configuration
$\cD^\pr=A\cup B\cup C\cup D^\pr$ in $\C P^1\times\C P^1$,
intersecting as depicted in Figure~\ref{fig:admisscase2}. Here
$D^\pr$ is a cuspidal rational curve of bi-degree $(2,2)$ and
$A,B,C$ are ruling fibres. (Note that $D$ necessarily intersects the
string $C_1\cup\cdots\cup C_k$ twice in an admissible
configuration).

\begin{figure}
\begin{center}\epsfig{file=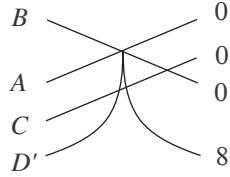}\end{center}
\caption{Arrangement of curves in $\C P^1\times\C P^1$ giving rise
to admissible configurations for case~II symplectic fillings
of links of tetrahedral, octahedral or icosahedral singularities of
type $(3,1)$} \label{fig:admisscase2}
\end{figure}

Assume now that the complement of the configuration $\cC^{\pr}=A\cup
B^\pr\cup D\cup C_1^\pr\cup\cdots\cup C_k^\pr$ in $M^{\pr}$ is
minimal and let $\cJ_{\cC^{\pr}}$ denote the set of compatible
almost complex structures with respect to which $\cC^{\pr}$ is
pseudoholomorphic.

\begin{proposition} \label{prop:case2} Let $J$ be a compatible almost
complex structure which is generic in $\cJ_{\cC^{\pr}}$. Denote by
$M^{\pr\pr}$ the symplectic $4$-manifold obtained by blowing down
all $J$-holomorphic $(-1)$-curves in $M^{\pr}\setminus (A\cup
B^\pr)$, by $D^\pr$ the image of $D$ and by $C_l^{\pr\pr}$ the image
of $C_l^\pr$. Then $\{A,B^\pr,D^\pr,\{C_l^{\pr\pr}\}\}$ is an
admissible configuration for a case II symplectic filling of a link
of a tetrahedral, octahedral or icosahedral singularity of type
$(3,1)$.
\end{proposition}

\begin{proof} {\bf Claim 1.} There is a $J$-holomorphic $(-1)$-curve $F$ in
$M^{\pr}\setminus (A\cup B^\pr)$ such that $F\cdot D=1$, $F\cdot
C_j^\pr=1$ for some $j$, $F\cdot C_l^\pr=0$, $l\neq j$.
\smallskip

\noindent {\it Proof of Claim 1.} Suppose that there is no such
curve $F$, then all $J$-holomorphic $(-1)$-curves in
$M^{\pr}\setminus (A\cup B^\pr)$ are disjoint from $D$. After
blowing down all such $(-1)$-curves, denote the image of $D$ by
$D^\pr$ and the image of $C_l^\pr$ by $C_l^{\pr\pr}$. By
Lemma~\ref{lem:case2}, there exists a $(-1)$-curve away from $A\cup
B^\pr$ in the resulting symplectic 4-manifold $M^{\pr\pr}$. Arguing
as in other cases, any such $(-1)$-curve must be one of the
$C_l^{\pr\pr}$, $l\neq i$. After iteratively blowing down all such
$(-1)$-curves we arrive at the situation depicted in
Figure~\ref{fig:step1}. But this situation can not be minimal since
if it were $C_i^{\pr\pr\pr}$ would have to be homologous to $A$ and
hence would have to intersect $D^{\pr\pr}$ twice. Hence, again by
Lemma~\ref{lem:case2}, there must be a $(-1)$-curve in the resulting
symplectic 4-manifold $M^{\pr\pr\pr}$ away from $A\cup B^\pr$. But
this $(-1)$-curve must have already existed in $M^\pr\setminus
(A\cup B^\pr)$, which is a contradiction since we are assuming that
we blew down all such $(-1)$-curves.

\begin{figure}
\begin{center}\epsfig{file=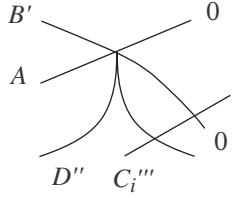}\end{center}
\caption{See Proof of Claim 1} \label{fig:step1}
\end{figure}

\smallskip \noindent
{\bf Claim 2.} Let $F$ be as in Claim~1. Then $j\geq i$.

\smallskip
\noindent {\it Proof of Claim 2.} Suppose $j<i$, then, after blowing
down all $(-1)$-curves in $M^{\pr}\setminus (A\cup B^\pr)$, let
$M^{\pr\pr}$ denote the resulting symplectic 4-manifold and denote
by $D^\pr$ the image of $D$ and by $C_l^{\pr\pr}$ the image of
$C_l^\pr$. Arguing as in the proof of Claim~1, we can now
iteratively blow down all the curves $C_l^{\pr\pr}$, $l\neq i$. Let
$M^{\pr\pr\pr}$ denote the resulting symplectic 4-manifold and
denote by $D^{\pr\pr}$ the image of $D^{\pr}$. Since the image of
the string $C_1^\pr,\ldots,C_{i-1}^\pr$ in $M^{\pr\pr}$ intersects
$D^\pr$ at least twice, $D^{\pr\pr}$ will have a singular point away
from the cusp point. Now, after contracting a maximal family of
pseudoholomorphic $(-1)$-curves in $M^{\pr\pr\pr}\setminus (A\cup
B^{\pr})$, we obtain a (2,2)-curve in $\C P^1\times\C P^1$ with at
least two singular points, which is impossible.

\smallskip \noindent
{\bf Claim 3.} There is at most one such curve $F$ as in Claim~1.

\smallskip
\noindent {\it Proof of Claim 3.} Suppose there is more than one
such curve and denote these curves $F_1,\ldots,F_s$. Assume that
$F_l\cdot C_{j(l)}^\pr=1$ for $l=1,\ldots, s$. Then, by Claim~2,
$j(l)\geq i$ for all $l$. First suppose that
$\sharp\{l\,|\,j(l)>i\}\geq 2$. Then after contracting all
$(-1)$-curves in $M^{\pr}\setminus (A\cup B^\pr)$, the image of the
string $C_{i+1}^\pr,\ldots C_k^\pr$ in the resulting symplectic
4-manifold $M^{\pr\pr}$ intersects the image of $D$ at least twice.
The proof is now as in the proof of Claim~2. Now suppose that
$\sharp\{l\,|\,j(l)>i\}\leq 1$. Then, after contracting all
$(-1)$-curves in $M^{\pr}\setminus (A\cup B^\pr)$ and then
iteratively contracting the images of the curves $C_l^\pr$, $l\neq
i$, denote by $M^{\pr\pr\pr}$ the resulting symplectic 4-manifold,
by $D^{\pr\pr}$ the image of $D$ and by $C_i^{\pr\pr\pr}$ the image
of $C_i^\pr$ (see Figure~\ref{fig:step3} for the case $s=2$). Note
that this situation does not occur in $\C P^1\times \C P^1$ since
$C_i^{\pr\pr\pr}$ is a smoothly embedded rational curve which is
disjoint from $A$ and hence must be homologous to $A$ but
$C_i^{\pr\pr\pr}\cdot D^{\pr\pr}\geq 3$. Since blowing down
$(-1)$-curves away from $A\cup B^\pr$ can only increase the
intersection number $C_i^{\pr\pr\pr}\cdot D^{\pr\pr}$, it follows
that $M^{\pr\pr\pr}$ can also not be a blow-up of $\C P^1\times \C
P^1$, which is absurd.

\smallskip
We prove the proposition. After contracting all $(-1)$-curves in
$M^\pr\setminus (A\cup B^\pr)$ it follows from Claim~3 that the
image of the string $C_1^\pr,\ldots,C_k^\pr$ intersects the image of
$D$ exactly twice. Namely, $C_1^{\pr\pr}$ and $C_j^{\pr\pr}$ for
some $j\geq i$ intersect $D^\pr$. One can now, using
Lemma~\ref{lem:case2}, iteratively blow down the curves
$C_l^{\pr\pr}$, $l\neq i$ to obtain the configuration given in
Figure~\ref{fig:admisscase2}. This shows that the configuration
$\{A,B^\pr,D^\pr,\{C_l^{\pr\pr}\}\}$ is an admissible configuration.
\qed
\end{proof}

\begin{figure}
\begin{center}\epsfig{file=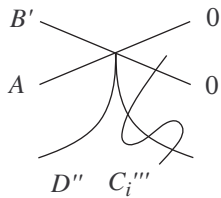}\end{center}
\caption{See Proof of Claim 3} \label{fig:step3}
\end{figure}


\section{Conclusion}
\label{sec:5} In Section~\ref{sec:4}, we reduced the
compactification to a standard configuration of rational curves in
either $\C P^2$ or $\C P^1 \times \C P^1$ as follows.
\begin{enumerate}[\rm (1)]
\item Cyclic quotient singularities: two distinct lines in $\C P^2$.
\item Dihedral singularities: a union of a $0$-curve and a
cuspidal curve of bi-degree $(2,2)$) in $\C P^1 \times\C P^1$ which
intersect at the cusp point, or the proper transform of a union of a
line and a cuspidal cubic curve in $\C P^2$ which meet at the cusp
point under the blow-up of their transversal intersection point,
that is, another intersection point of them.
\item Tetrahedral, Octahedral and Icosahedral singularities of type
$(3,2)$: a cuspidal curve of degree 3 (resp.\ of bi-degree $(2,2)$)
in $\C P^2$ (resp.\ $\C P^1 \times \C P^1$).
\item Tetrahedral, Octahedral and Icosahedral singularities of type
$(3,1)$:  two kinds of configurations appear:
\begin{enumerate}[\rm (i)]
\item a union of $\C P^1 \times \{\pt\}$, $\{\pt \} \times \C P^1$ and
a cuspidal curve of bi-degree $(2,2)$, which meet at the cusp point,
\item a union of $\C P^1 \times \{\pt\}$, $\{\pt \} \times \C
P^1$, a cuspidal curve of bi-degree $(2,2)$, which meet at the cusp
point, and another rational curve homologous to $\C P^1 \times
\{\pt\}$.
\end{enumerate}
\end{enumerate}

To recover the symplectic filling $X$, we first sequentially blow up
the manifold at points on the total transform of the divisor in the
above list to get a closed symplectic 4-manifold $Z$. Then we get
$X$ as the complement of a regular neighbourhood of the
compactifying divisor $K$ in $Z$. For classification up to
symplectic deformation equivalence, we need uniqueness of symplectic
deformation types of the standard configurations, which we can prove
as in \cite{OO2}.

It is not difficult to see that for cyclic quotient singularities
and dihedral singularities we can find links with arbitrarily many
nondiffeomorphic symplectic fillings. (For cyclic quotient
singularities, this fact was noted by Lisca~\cite{L}.) However for
tetrahedral, octahedral and icosahedral singularities, the number of
symplectic fillings for each class of singularities in Tables 1, 2,
3 and 4 is bounded above by a number independent of $b$. We give a
list of all symplectic fillings in these cases. To aid this, for the
case (4ii) above, we note the following constraints:
\begin{enumerate}[\rm (a)]
\item if $i>1$, then $c_i\neq 2$,
\item if $j<k$, then $c_j\neq 2$,
\item $b\leq \max\{5,c_{b-2}\}$,
\end{enumerate}
where $i$ is as in Figure~\ref{fig:case2} and $j$ is as in the proof
of Proposition~\ref{prop:case2}. In particular, since $c_{b-2}\leq
6$ for quotient singularities, there are only a finite number of
symplectic filling which fall into case (4ii) above.

The list we give below is the list of compactifications $Z$ and
compactifying divisors $K$ of minimal symplectic fillings.  There
may be symplectically deformation equivalent fillings in the list.
To get a list of minimal symplectic fillings, we should describe the
contactomorphisms up to contact isotopies. We leave this as a topic
for future research.

\subsubsection*{{\it Symplectic fillings of links of Tetrahedral,
Octahedral and Icosahedral singularities of type $(3,2)$}}

We use the notation $(m;D\cdot D,-c_1,\ldots,-c_k;a_1\times
i_1,\ldots,a_l\times i_l)$ to denote the symplectic filling of the
link of $T_m,O_m$ or $I_m$ given as the complement of a regular
neighbourhood of the compactifying divisor $K=D\cup
C_1\cup\cdots\cup C_k$ given in Figure~\ref{fig:2-2blconfig}. Here
$-c_1,\ldots,-c_k$ denote the self-intersections of the curves
$C_1,\ldots,C_k$ and $a_j\times i_j$ denotes the existence of $a_j$
distinct $(-1)$-curves intersecting $C_{i_j}$ in $Z$. We abbreviate
$1\times i_j=i_j$. In each case we indicate whether the pair $(Z,K)$
is given by blowing up $(\C P^2,$ cuspidal cubic curve of degree
$3)$ or $(\C P^1\times \C P^1,$ cuspidal curve of bi-degree
$(2,2))$. Note that when we blow down the compactification $Z$ of a
symplectic filling of the link of a singularity of type $(3,2)$ we
can always guarantee that we end up with $\C P^2$ unless the image
of $D$ under the blowing down map has self-intersection number 8. In
that case we may also end up with $\C P^1\times \C P^1$. There are
16 cases where this occurs.

\subsubsection*{{\it Tetrahedral, $T_m$}}
$\mbox{}$\\
1. $E_6$\\
2. $(6(3-2)+1;5,-4;3\times 1),\,\C P^2$\\
3. $(6(4-2)+1;5,-2,-4;3\times 2),\,\C P^2$\\
4. $(6(4-2)+1;5,-2,-4;1,2\times 2),\,\C P^2$\\
5. $(6(5-2)+1;5,-2,-2,-4;3\times 3),\,\C P^2$\\
6. $(6(5-2)+1;5,-2,-2,-4;1,2\times 3),\,\C P^2$\\
7. $(6(5-2)+1;5,-2,-2,-4;1,2\times 3),\,\C P^1\times\C P^1$\\
8. $(6(5-2)+1;5,-2,-2,-4;2,3),\,\C P^2$\\
9. $(6(6-2)+1;5,-2,-2,-2,-4;3\times 4),\,\C P^2$\\
10. $(6(6-2)+1;5,-2,-2,-2,-4;1,2\times 4),\,\C P^2$\\
11. $(6(6-2)+1;5,-2,-2,-2,-4;3),\,\C P^2$\\
12. $(6(b-2)+1,b\geq 7;5,\underbrace{-2,\ldots,-2}_{b-3},-4;3\times k),\,k=b-2,\,\C P^2$\\
13. $(6(2-2)+3;4,-2;1),\,\C P^2$\\
14. $(6(3-2)+3;5,-3,-2;1,2),\,\C P^2$\\
15. $(6(3-2)+3;5,-3,-2;2\times 1),\,\C P^2$\\
16. $(6(4-2)+3;5,-2,-3,-2;2,3),\,\C P^2$\\
17. $(6(4-2)+3;5,-2,-3,-2;1,3),\,\C P^2$\\
18. $(6(4-2)+3;5,-2,-3,-2;1,2),\,\C P^2$\\
19. $(6(4-2)+3;5,-2,-3,-2;1,2),\,\C P^1\times\C P^1$\\
20. $(6(5-2)+3;5,-2,-2,-3,-2;3,4),\,\C P^2$\\
21. $(6(5-2)+3;5,-2,-2,-3,-2;1,4),\,\C P^2$\\
22. $(6(5-2)+3;5,-2,-2,-3,-2;1,4),\,\C P^1\times\C P^1$\\
23. $(6(5-2)+3;5,-2,-2,-3,-2;1,3),\,\C P^2$\\
24. $(6(5-2)+3;5,-2,-2,-3,-2;2),\,\C P^2$\\
25. $(6(5-2)+3;5,-2,-2,-3,-2;2),\,\C P^1\times\C P^1$\\
26. $(6(6-2)+3;5,-2,-2,-2,-3,-2;4,5),\,\C P^2$\\
27. $(6(6-2)+3;5,-2,-2,-2,-3,-2;1,5),\,\C P^2$\\
28. $(6(b-2)+3,b\geq
7;5,\underbrace{-2,\ldots,-2}_{b-3},-3,-2;k-1,k),\,k=b-1,\,\C P^2$

\subsubsection*{{\it Octahedral, $O_m$}}
$\mbox{}$\\
29. $E_7$\\
30. $(12(3-2)+1;5,-5;4\times 1),\,\C P^2$\\
31. $(12(4-2)+1;5,-2,-5;4\times 2),\,\C P^2$\\
32. $(12(4-2)+1;5,-2,-5;1,3\times 2),\,\C P^2$\\
33. $(12(5-2)+1;5,-2,-2,-5;4\times 3),\,\C P^2$\\
34. $(12(5-2)+1;5,-2,-2,-5;1,3\times 3),\,\C P^2$\\
35. $(12(5-2)+1;5,-2,-2,-5;1,3\times 3),\,\C P^1\times\C P^1$\\
36. $(12(5-2)+1;5,-2,-2,-5;2,2\times 3),\,\C P^2$\\
37. $(12(6-2)+1;5,-2,-2,-2,-5;4\times 4),\,\C P^2$\\
38. $(12(6-2)+1;5,-2,-2,-2,-5;1,3\times 4),\,\C P^2$\\
39. $(12(6-2)+1;5,-2,-2,-2,-5;3,4),\,\C P^2$\\
40. $(12(7-2)+1;5,-2,-2,-2,-2,-5;4\times 5),\,\C P^2$\\
41. $(12(7-2)+1;5,-2,-2,-2,-2,-5;4),\,\C P^2$\\
42. $(12(b-2)+1,b\geq 8;5,\underbrace{-2,\ldots,-2}_{b-3},-5;4\times k),\,k=b-2,\,\C P^2$\\
43. $(12(2-2)+7;4,-2,-2;2),\,\C P^2$\\
44. $(12(2-2)+7;4,-2,-2;1),\,\C P^2$\\
45. $(12(3-2)+7;5,-3,-2,-2;1,3),\,\C P^2$\\
46. $(12(3-2)+7;5,-3,-2,-2;2\times 1),\,\C P^2$\\
47. $(12(3-2)+7;5,-3,-2,-2;2\times 1),\,\C P^1\times\C P^1$\\
48. $(12(3-2)+7;5,-3,-2,-2;2),\,\C P^2$\\
49. $(12(4-2)+7;5,-2,-3,-2,-2;2,4),\,\C P^2$\\
50. $(12(4-2)+7;5,-2,-3,-2,-2;1,2),\,\C P^2$\\
51. $(12(4-2)+7;5,-2,-3,-2,-2;1,4),\,\C P^2$\\
52. $(12(4-2)+7;5,-2,-3,-2,-2;3),\,\C P^2$\\
53. $(12(5-2)+7;5,-2,-2,-3,-2,-2;3,5),\,\C P^2$\\
54. $(12(5-2)+7;5,-2,-2,-3,-2,-2;1,5),\,\C P^2$\\
55. $(12(5-2)+7;5,-2,-2,-3,-2,-2;1,5),\,\C P^1\times\C P^1$\\
56. $(12(5-2)+7;5,-2,-2,-3,-2,-2;4),\,\C P^2$\\
57. $(12(5-2)+7;5,-2,-2,-3,-2,-2;2),\,\C P^2$\\
58. $(12(6-2)+7;5,-2,-2,-2,-3,-2,-2;4,6),\,\C P^2$\\
59. $(12(6-2)+7;5,-2,-2,-2,-3,-2,-2;1,6),\,\C P^2$\\
60. $(12(6-2)+7;5,-2,-2,-2,-3,-2,-2;5),\,\C P^2$\\
61. $(12(b-2)+7,b\geq 7;5,\underbrace{-2,\ldots,-2}_{b-3},-3,-2,-2;k-2,k),\,k=b,\,\C P^2$\\
62. $(12(b-2)+7,b\geq
7;5,\underbrace{-2,\ldots,-2}_{b-3},-3,-2,-2;k-1),\,k=b,\,\C P^2$

\subsubsection*{{\it Icosahedral, $I_m$}}
$\mbox{}$\\
63. $E_8$\\
64. $(30(3-2)+1;5,-6;5\times 1),\,\C P^2$\\
65. $(30(4-2)+1;5,-2,-6;5\times 2),\,\C P^2$\\
66. $(30(4-2)+1;5,-2,-6;1,4\times 2),\,\C P^2$\\
67. $(30(5-2)+1;5,-2,-2,-6;5\times 3),\,\C P^2$\\
68. $(30(5-2)+1;5,-2,-2,-6;1,4\times 3),\,\C P^2$\\
69. $(30(5-2)+1;5,-2,-2,-6;1,4\times 3),\,\C P^1\times\C P^1$\\
70. $(30(5-2)+1;5,-2,-2,-6;2,3\times 3),\,\C P^2$\\
71. $(30(6-2)+1;5,-2,-2,-2,-6;5\times 4),\,\C P^2$\\
72. $(30(6-2)+1;5,-2,-2,-2,-6;1,4\times 4),\,\C P^2$\\
73. $(30(6-2)+1;5,-2,-2,-2,-6;3,2\times 4),\,\C P^2$\\
74. $(30(7-2)+1;5,-2,-2,-2,-2,-6;5\times 5),\,\C P^2$\\
75. $(30(7-2)+1;5,-2,-2,-2,-2,-6;4,5)$ $\C P^2$\\
76. $(30(8-2)+1;5,-2,-2,-2,-2,-2,-6;5\times 6),\,\C P^2$\\
77. $(30(8-2)+1;5,-2,-2,-2,-2,-2,-6;5)$ $\C P^2$\\
78. $(30(b-2)+1,b\geq 9;5,\underbrace{-2,\ldots,-2}_{b-3},-6;5\times k),\,k=b-2,\,\C P^2$\\
79. $(30(2-2)+7;3,-2;1),\,\C P^2$\\
80. $(30(3-2)+7;5,-4,-2;2\times 1,2),\,\C P^2$\\
81. $(30(3-2)+7;5,-4,-2;3\times 1),\,\C P^2$\\
82. $(30(4-2)+7;5,-2,-4,-2;2\times 2,3),\,\C P^2$\\
83. $(30(4-2)+7;5,-2,-4,-2;1,2\times 2),\,\C P^2$\\
84. $(30(4-2)+7;5,-2,-4,-2;1,2\times 2),\,\C P^1\times\C P^1$\\
85. $(30(4-2)+7;5,-2,-4,-2;1,2,3),\,\C P^2$\\
86. $(30(5-2)+7;5,-2,-2,-4,-2;2\times 3,4),\,\C P^2$\\
87. $(30(5-2)+7;5,-2,-2,-4,-2;1,2\times 3),\,\C P^2$\\
88. $(30(5-2)+7;5,-2,-2,-4,-2;2,3),\,\C P^2$\\
89. $(30(5-2)+7;5,-2,-2,-4,-2;2,3),\,\C P^1\times\C P^1$\\
90. $(30(5-2)+7;5,-2,-2,-4,-2;2,4),\,\C P^2$\\
91. $(30(5-2)+7;5,-2,-2,-4,-2;1,3,4),\,\C P^2$\\
92. $(30(5-2)+7;5,-2,-2,-4,-2;1,3,4),\,\C P^1\times\C P^1$\\
93. $(30(6-2)+7;5,-2,-2,-2,-4,-2;2\times 4,5),\,\C P^2$\\
94. $(30(6-2)+7;5,-2,-2,-2,-4,-2;1,4,5),\,\C P^2$\\
95. $(30(6-2)+7;5,-2,-2,-2,-4,-2;3),\,\C P^2$\\
96. $(30(6-2)+7;5,-2,-2,-2,-4,-2;3),\,\C P^1\times\C P^1$\\
97. $(30(b-2)+7,b\geq 7;5,\underbrace{-2,\ldots,-2}_{b-3},-4,-2;2\times (k-1),k),\,k=b-1,\,\C P^2$\\
98. $(30(2-2)+13;4,-3;2\times 1),\,\C P^2$\\
99. $(30(3-2)+13;5,-3,-3;1,2\times 2),\,\C P^2$\\
100. $(30(3-2)+13;5,-3,-3;2\times 1,2),\,\C P^2$\\
101. $(30(4-2)+13;5,-2,-3,-3;2,2\times 3),\,\C P^2$\\
102. $(30(4-2)+13;5,-2,-3,-3;1,2,3),\,\C P^2$\\
103. $(30(4-2)+13;5,-2,-3,-3;1,2,3),\,\C P^1\times\C P^1$\\
104. $(30(4-2)+13;5,-2,-3,-3;1,2\times 3),\,\C P^2$\\
105. $(30(4-2)+13;5,-2,-3,-3;2\times 2),\,\C P^2$\\
106. $(30(5-2)+13;5,-2,-2,-3,-3;3,2\times 4),\,\C P^2$\\
107. $(30(5-2)+13;5,-2,-2,-3,-3;1,2\times 4),\,\C P^2$\\
108. $(30(5-2)+13;5,-2,-2,-3,-3;1,2\times 4),\,\C P^1\times\C P^1$\\
109. $(30(5-2)+13;5,-2,-2,-3,-3;1,3,4),\,\C P^2$\\
110. $(30(5-2)+13;5,-2,-2,-3,-3;2,4),\,\C P^2$\\
111. $(30(5-2)+13;5,-2,-2,-3,-3;2,4),\,\C P^1\times\C P^1$\\
112. $(30(6-2)+13;5,-2,-2,-2,-3,-3;4,2\times 5),\,\C P^2$\\
113. $(30(6-2)+13;5,-2,-2,-2,-3,-3;1,2\times 5),\,\C P^2$\\
114. $(30(b-2)+13,b\geq 7;5,\underbrace{-2,\ldots,-2}_{b-3},-3,-3;k-1,2\times k),\,k=b-1,\,\C P^2$\\
115. $(30(2-2)+19;4,-2,-2,-2;3),\,\C P^2$\\
116. $(30(2-2)+19;4,-2,-2,-2;1),\,\C P^2$\\
117. $(30(3-2)+19;5,-3,-2,-2,-2;1,4),\,\C P^2$\\
118. $(30(3-2)+19;5,-3,-2,-2,-2;2\times 1),\,\C P^2$\\
119. $(30(4-2)+19;5,-2,-3,-2,-2,-2;2,5),\,\C P^2$\\
120. $(30(4-2)+19;5,-2,-3,-2,-2,-2;1,5),\,\C P^2$\\
121. $(30(5-2)+19;5,-2,-2,-3,-2,-2,-2;3,6),\,\C P^2$\\
122. $(30(5-2)+19;5,-2,-2,-3,-2,-2,-2;1,6),\,\C P^2$\\
123. $(30(5-2)+19;5,-2,-2,-3,-2,-2,-2;1,6),\,\C P^1\times\C P^1$\\
124. $(30(6-2)+19;5,-2,-2,-2,-3,-2,-2,-2;4,7),\,\C P^2$\\
125. $(30(6-2)+19;5,-2,-2,-2,-3,-2,-2,-2;1,7),\,\C P^2$\\
126. $(30(b-2)+19,b\geq
7;5,\underbrace{-2,\ldots,-2}_{b-3},-3,-2,-2,-2;k-3,k), k=b+1,\,\C
P^2$

\subsubsection*{{\it Symplectic fillings of links of Tetrahedral,
Octahedral and Icosahedral singularities of type $(3,1)$}}

\subsubsection*{{\it Case I.}} Here refer we to final picture in
Figure~\ref{fig:3-bl}. We use the notation $(m;D\cdot
D,-c_1,\ldots,-c_k;a_1\times i_1,\ldots,a_l\times i_l)$ to denote
the case~I symplectic filling of the link of $T_m,O_m$ or
$I_m$ given as the complement of a regular neighbourhood of the
compactifying divisor $K=A\cup B\cup D\cup C_1\cup\cdots\cup C_k$ in
$Z$. The notation is as for symplectic fillings of links of
singularities of type $(3,2)$.

\subsubsection*{{\it Tetrahedral, $T_m$}}
$\mbox{}$\\
127. $(6(2-2)+5;4,-2;1)$\\
128. $(6(3-2)+5;5,-3,-2;1,2)$\\
129. $(6(3-2)+5;5,-3,-2;2\times 1)$\\
130. $(6(4-2)+5;5,-2,-3,-2;2,3)$\\
131. $(6(4-2)+5;5,-2,-3,-2;1,3)$\\
132. $(6(4-2)+5;5,-2,-3,-2;1,2)$\\
133. $(6(5-2)+5;5,-2,-2,-3,-2;3,4)$\\
134. $(6(5-2)+5;5,-2,-2,-3,-2;1,4)$\\
135. $(6(5-2)+5;5,-2,-2,-3,-2;2)$\\
136. $(6(b-2)+5,b\geq
6;5,\underbrace{-2,\ldots,-2}_{b-3},-3,-2;k-1,k), k=b-1$

\subsubsection*{{\it Octahedral, $O_m$}}
$\mbox{}$\\
137. $(12(2-2)+5;2;)$\\
138. $(12(3-2)+5;5,-5;4\times 1)$\\
139. $(12(4-2)+5;5,-2,-5;4\times 2)$\\
140. $(12(4-2)+5;5,-2,-5;1,3\times 2)$\\
141. $(12(5-2)+5;5,-2,-2,-5;4\times 3)$\\
142. $(12(5-2)+5;5,-2,-2,-5;1,3\times 3)$\\
143. $(12(5-2)+5;5,-2,-2,-5;2,2\times 3)$\\
144. $(12(6-2)+5;5,-2,-2,-2,-5;4\times 4)$\\
145. $(12(6-2)+5;5,-2,-2,-2,-5;3,4)$\\
146. $(12(7-2)+5;5,-2,-2,-2,-2,-5;4\times 5)$\\
147. $(12(7-2)+5;5,-2,-2,-2,-2,-5;4)$\\
148. $(12(b-2)+5,b\geq 8;5,\underbrace{-2,\ldots,-2}_{b-3},-5;4\times k), k=b-2$\\
149. $(12(2-2)+11;4,-2,-2;2)$\\
150. $(12(2-2)+11;4,-2,-2;1)$\\
151. $(12(3-2)+11;5,-3,-2,-2;1,3)$\\
152. $(12(3-2)+11;5,-3,-2,-2;2\times 1)$\\
153. $(12(3-2)+11;5,-3,-2,-2;2)$\\
154. $(12(4-2)+11;5,-2,-3,-2,-2;2,4)$\\
155. $(12(4-2)+11;5,-2,-3,-2,-2;1,4)$\\
156. $(12(4-2)+11;5,-2,-3,-2,-2;3)$\\
157. $(12(5-2)+11;5,-2,-2,-3,-2,-2;3,5)$\\
158. $(12(5-2)+11;5,-2,-2,-3,-2,-2;1,5)$\\
159. $(12(5-2)+11;5,-2,-2,-3,-2,-2;4)$\\
160. $(12(b-2)+11,b\geq 6;5,\underbrace{-2,\ldots,-2}_{b-3},-3,-2,-2;k-2,k), k=b$\\
161. $(12(b-2)+11,b\geq
6;5,\underbrace{-2,\ldots,-2}_{b-3},-3,-2,-2;k-1), k=b$

\subsubsection*{{\it Icosahedral, $I_m$}}
$\mbox{}$\\
162. $(30(2-2)+11;1;)$\\
163. $(30(3-2)+11;5,-6;5\times 1)$\\
164. $(30(4-2)+11;5,-2,-6;5\times 2)$\\
165. $(30(4-2)+11;5,-2,-6;1,4\times 2)$\\
166. $(30(5-2)+11;5,-2,-2,-6;5\times 3)$\\
167. $(30(5-2)+11;5,-2,-2,-6;1,4\times 3)$\\
168. $(30(5-2)+11;5,-2,-2,-6;2,3\times 3)$\\
169. $(30(6-2)+11;5,-2,-2,-2,-6;5\times 4)$\\
170. $(30(6-2)+11;5,-2,-2,-2,-6;3,2\times 4)$\\
171. $(30(7-2)+11;5,-2,-2,-2,-2,-6;5\times 5)$\\
172. $(30(7-2)+11;5,-2,-2,-2,-2,-6;4,5)$\\
173. $(30(8-2)+11;5,-2,-2,-2,-2,-2,-6;5\times 6)$\\
174. $(30(8-2)+11;5,-2,-2,-2,-2,-2,-6;5)$\\
175. $(30(b-2)+11,b\geq 9;5,\underbrace{-2,\ldots,-2}_{b-3},-6;5\times k), k=b-2$\\
176. $(30(2-2)+17;3,-2;1)$\\
177. $(30(3-2)+17;5,-4,-2;2\times 1,2)$\\
178. $(30(3-2)+17;5,-4,-2;3\times 1)$\\
179. $(30(4-2)+17;5,-2,-4,-2;2\times 2,3)$\\
180. $(30(4-2)+17;5,-2,-4,-2;1,2\times 2)$\\
181. $(30(4-2)+17;5,-2,-4,-2;1,2,3)$\\
182. $(30(5-2)+17;5,-2,-2,-4,-2;2\times 3,4)$\\
183. $(30(5-2)+17;5,-2,-2,-4,-2;2,3)$\\
184. $(30(5-2)+17;5,-2,-2,-4,-2;2,4)$\\
185. $(30(5-2)+17;5,-2,-2,-4,-2;1,3,4)$\\
186. $(30(6-2)+17;5,-2,-2,-2,-4,-2;2\times 4,5)$\\
187. $(30(6-2)+17;5,-2,-2,-2,-4,-2;3)$\\
188. $(30(b-2)+17,b\geq 7;5,\underbrace{-2,\ldots,-2}_{b-3},-4,-2;2\times (k-1),k), k=b-1$\\
189. $(30(2-2)+23;4,-3;2\times 1)$\\
190. $(30(3-2)+23;5,-3,-3;1,2\times 2)$\\
191. $(30(3-2)+23;5,-3,-3;2\times 1,2)$\\
192. $(30(4-2)+23;5,-2,-3,-3;2,2\times 3)$\\
193. $(30(4-2)+23;5,-2,-3,-3;1,2,3)$\\
194. $(30(4-2)+23;5,-2,-3,-3;1,2\times 3)$\\
195. $(30(4-2)+23;5,-2,-3,-3;2\times 2)$\\
196. $(30(5-2)+23;5,-2,-2,-3,-3;3,2\times 4)$\\
197. $(30(5-2)+23;5,-2,-2,-3,-3;1,2\times 4)$\\
198. $(30(5-2)+23;5,-2,-2,-3,-3;2,4)$\\
199. $(30(b-2)+23,b\geq 6;5,\underbrace{-2,\ldots,-2}_{b-3},-3,-3;k-1,2\times k), k=b-1$\\
200. $(30(2-2)+29;4,-2,-2,-2;3)$\\
201. $(30(2-2)+29;4,-2,-2,-2;1)$\\
202. $(30(3-2)+29;5,-3,-2,-2,-2;1,4)$\\
203. $(30(4-2)+29;5,-2,-3,-2,-2,-2;2,5)$\\
204. $(30(4-2)+29;5,-2,-3,-2,-2,-2;1,5)$\\
205. $(30(5-2)+29;5,-2,-2,-3,-2,-2,-2;3,6)$\\
206. $(30(5-2)+29;5,-2,-2,-3,-2,-2,-2;1,6)$\\
207. $(30(b-2)+29,b\geq
7;5,\underbrace{-2,\ldots,-2}_{b-3},-3,-2,-2,-2;k-3,k), k=b+1$

\subsubsection*{{\it Case II.}} Again we refer to the final picture in
Figure~\ref{fig:3-bl}. We use the notation $(m;D\cdot
D,-c_1,\ldots,-c_k;i,j;a_1\times i_1,\ldots,a_l\times i_l)$ to
denote the case~II symplectic filling of the link of $T_m,O_m$
or $I_m$ given as the complement of a regular neighbourhood of the
compactifying divisor $K=A\cup B\cup D\cup C_1\cup\cdots\cup C_k$ in
$Z$. Here the numbers $i$ and $j$ denote the existence of
$(-1)$-curves intersecting $B$ and $C_i$ and $D$ and $C_j$
respectively.

\subsubsection*{{\it Tetrahedral, $T_m$}}
$\mbox{}$\\
208. $(6(2-2)+5;4,-2;1,1;)$\\
209. $(6(3-2)+5;5,-3,-2;1,1;2)$\\
210. $(6(3-2)+5;5,-3,-2;1,2;1)$\\
211. $(6(4-2)+5;5,-2,-3,-2;1,2;3)$\\
212. $(6(4-2)+5;5,-2,-3,-2;1,3;2)$\\
213. $(6(4-2)+5;5,-2,-3,-2;2,3;1)$\\
214. $(6(5-2)+5;5,-2,-2,-3,-2;1,3;4)$

\subsubsection*{{\it Octahedral, $O_m$}}
$\mbox{}$\\
215. $(12(3-2)+5;5,-5;1,1;3\times 1)$\\
216. $(12(4-2)+5;5,-2,-5;1,2;3\times 2)$\\
217. $(12(4-2)+5;5,-2,-5;2,2;1,2\times 2)$\\
218. $(12(5-2)+5;5,-2,-2,-5;1,3;3\times 3)$\\
219. $(12(5-2)+5;5,-2,-2,-5;3,3;1,2\times 3)$\\
220. $(12(5-2)+5;5,-2,-2,-5;3,3;2,3)$\\
221. $(12(6-2)+5;5,-2,-2,-2,-5;4,4;3)$\\
222. $(12(2-2)+11;4,-2,-2;1,2;)$\\
223. $(12(3-2)+11;5,-3,-2,-2;1,1;3)$\\
224. $(12(3-2)+11;5,-3,-2,-2;1,3;1)$\\
225. $(12(4-2)+11;5,-2,-3,-2,-2;1,2;4)$\\
226. $(12(5-2)+11;5,-2,-2,-3,-2,-2;1,3;5)$

\subsubsection*{{\it Icosahedral, $I_m$}}
$\mbox{}$\\
227. $(30(3-2)+11;5,-6;1,1;4\times 1)$\\
228. $(30(4-2)+11;5,-2,-6;1,2;4\times 2)$\\
229. $(30(4-2)+11;5,-2,-6;2,2;1,3\times 2)$\\
230. $(30(5-2)+11;5,-2,-2,-6;1,3;4\times 3)$\\
231. $(30(5-2)+11;5,-2,-2,-6;3,3;1,3\times 3)$\\
232. $(30(5-2)+11;5,-2,-2,-6;3,3;2,2\times 3)$\\
233. $(30(6-2)+11;5,-2,-2,-2,-6;4,4;3,4)$\\
234. $(30(7-2)+11;5,-2,-2,-2,-2,-6;5,5;4)$\\
235. $(30(2-2)+17;3,-2;1,1;)$\\
236. $(30(3-2)+17;5,-4,-2;1,1;1,2)$\\
237. $(30(3-2)+17;5,-4,-2;1,2;2\times 1)$\\
238. $(30(4-2)+17;5,-2,-4,-2;1,2;2,3)$\\
239. $(30(4-2)+17;5,-2,-4,-2;1,3;2\times 2)$\\
240. $(30(4-2)+17;5,-2,-4,-2;2,2;1,3)$\\
241. $(30(4-2)+17;5,-2,-4,-2;2,3;1,2)$\\
242. $(30(5-2)+17;5,-2,-2,-4,-2;1,3;3,4)$\\
243. $(30(5-2)+17;5,-2,-2,-4,-2;3,3;1,4)$\\
244. $(30(5-2)+17;5,-2,-2,-4,-2;3,4;2)$\\
245. $(30(2-2)+23;4,-3;1,1;1)$\\
246. $(30(3-2)+23;5,-3,-3;1,1;2\times 2)$\\
247. $(30(3-2)+23;5,-3,-3;1,2;1,2)$\\
248. $(30(3-2)+23;5,-3,-3;2,2;2\times 1)$\\
249. $(30(4-2)+23;5,-2,-3,-3;1,2;2\times 3)$\\
250. $(30(4-2)+23;5,-2,-3,-3;1,3;2,3)$\\
251. $(30(4-2)+23;5,-2,-3,-3;2,3;1,3)$\\
252. $(30(4-2)+23;5,-2,-3,-3;3,3;1,2)$\\
253. $(30(5-2)+23;5,-2,-2,-3,-3;1,3;2\times 4)$\\
254. $(30(2-2)+29;4,-2,-2,-2;1,3;)$\\
255. $(30(3-2)+29;5,-3,-2,-2,-2;1,1;4)$\\
256. $(30(4-2)+29;5,-2,-3,-2,-2,-2;1,2;5)$\\
257. $(30(5-2)+29;5,-2,-2,-3,-2,-2,-2;1,3;6)$

\section*{Acknowledgements} The main part of this work was carried out
during the first author's stay at Hokkaido University. This stay was
supported by a JSPS Postdoctoral Fellowship for Foreign Researchers
in Japan which the first author gratefully acknowledges.

{\sc Department of Mathematics, Middle East Technical University, 06531 Ankara, Turkey.}

E-mail address: {\tt bhupal@metu.edu.tr} 

\medskip
{\sc Department of Mathematics, Hokkaido University, Sapporo 060-0810,
Japan.} 

E-mail address: {\tt ono@math.sci.hokudai.ac.jp} 

\end{document}